\documentclass[11pt,a4paper,oneside,reqno]{amsart}

\usepackage{latexsym}
\usepackage{amsmath}
\usepackage{amsthm}
\usepackage{amssymb}
\usepackage{amsfonts}
\usepackage{fancyhdr}
\usepackage{comment}

\usepackage{mathrsfs}                           

\newcommand{\diag}{\mathrm{diag}}

\newcommand{\mR}{\mathbf{R}}                    
\newcommand{\mC}{\mathbf{C}}                    
\newcommand{\abs}[1]{\lvert #1 \rvert}          
\newcommand{\norm}[1]{\lVert #1 \rVert}         
\newcommand{\br}[1]{\langle #1 \rangle}         

\newcommand{\eps}{\varepsilon}

\newcommand{\mSp}{\mathscr{S}^{\prime}}

\newcommand{\mF}{\mathscr{F}}

\newcommand{\re}{\mathrm{Re}}

\newcommand{\supp}{\mathrm{supp}}

\newcommand{\closure}[1]{\overline{#1}}
\newcommand{\dbar}{\overline{\partial}}

\theoremstyle{definition}
\newtheorem{thm}{Theorem}[section]
\newtheorem{prop}[thm]{Proposition}

\newtheorem{lemma}[thm]{Lemma}

\newtheorem*{definition}{Definition}

\newtheorem*{remark}{Remark}

\newtheorem*{remarks}{Remarks}

\numberwithin{equation}{section}

\title{Inverse problems for the anisotropic Maxwell equations}
\author{Carlos E. Kenig}
\address{Department of Mathematics, University of Chicago}
\email{cek@math.uchicago.edu}
\author{Mikko Salo}
\address{Department of Mathematics and Statistics, University of Helsinki}
\email{mikko.salo@helsinki.fi}
\author{Gunther Uhlmann}
\address{Department of Mathematics, University of Washington}
\email{gunther@math.washington.edu}

\begin{document}

\begin{abstract}
We prove that the electromagnetic material parameters are uniquely determined by boundary measurements for the time-harmonic Maxwell equations in certain anisotropic settings. We give a uniqueness result in the inverse problem for Maxwell equations on an admissible Riemannian manifold, and a uniqueness result for Maxwell equations in Euclidean space with admissible matrix coefficients. The proofs are based on a new Fourier analytic construction of complex geometrical optics solutions on admissible manifolds, and involve a proper notion of uniqueness for such solutions.
\end{abstract}

\maketitle


\section{Introduction}

Let $(M,g)$ be a compact Riemannian manifold with smooth boundary $\partial M$, and assume that $\text{dim}\,M = 3$. We consider the inverse problem of recovering electromagnetic material parameters of the medium $(M,g)$ by probing with time-harmonic electromagnetic fields. The fields in $(M,g)$ are described by complex $1$-forms $E$ and $H$ (electric and magnetic fields), and the behavior of the fields is governed by the Maxwell equations in $M$, 
\begin{equation} \label{maxwell_equations}
\left\{ \begin{array}{rl}
*dE &\!\!\!= i\omega \mu H, \\
*dH &\!\!\!= -i\omega \eps E.
\end{array} \right.
\end{equation}
Here $\omega > 0$ is a fixed frequency, $d$ is the exterior derivative, and $*$ is the Hodge star operator on $(M,g)$. The material parameters are given by the complex functions $\eps$ and $\mu$ (permittivity and permeability, respectively). We assume the following conditions on the parameters:
\begin{eqnarray}
 & \text{$\eps, \mu \in C^{\infty}(M)$}, & \label{paramcond1} \\
 & \text{$\re(\eps) > 0$, $\re(\mu) > 0$ in $M$}. & \label{paramcond2}
\end{eqnarray}

For the inverse problem, we need to describe the electromagnetic field measurements at the boundary $\partial M$. Let $i: \partial M \to M$ be the canonical embedding, and consider the tangential trace on $k$-forms, 
\begin{equation*}
t: \Omega^k(M) \to \Omega^k(\partial M), \quad \eta \mapsto i^* \eta.
\end{equation*}
There is a discrete set of resonant frequencies such that if $\omega$ is outside this set, then for any $f$ in $\Omega^1(\partial M)$ the system \eqref{maxwell_equations} has a unique solution $(E,H)$ satisfying $tE = f$ (see Theorem \ref{thm:app_wellposedness}). We shall assume that 
\begin{equation} \label{paramcond3p}
\text{$\omega > 0$ is not a resonant frequency.}
\end{equation}

\newpage

The boundary measurements are given by the admittance map 
\begin{equation*}
\Lambda : \Omega^1(\partial M) \to \Omega^1(\partial M), \ tE \mapsto tH.
\end{equation*}
The inverse problem for time-harmonic Maxwell equations is to recover the material parameters $\eps$ and $\mu$ from the knowledge of the admittance map $\Lambda$.

In the case of lossy materials, one writes $\eps = \re(\eps) + i\sigma/\omega$ where $\sigma \geq 0$ is the conductivity. The zero frequency case (that is, $\omega = 0$) then formally corresponds to the conductivity equation 
\begin{equation*}
\delta(\sigma du) = 0.
\end{equation*}
Here $\delta$ is the codifferential. In three and higher dimensions, the inverse problem of determining $\sigma$ from boundary measurements for the conductivity equation was studied in \cite{DKSaU} in a special class of Riemannian manifolds.

\begin{definition}
A compact $3$-manifold $(M,g)$ with smooth boundary $\partial M$ is called \emph{admissible} if $(M,g)$ is embedded in $(T,g)$ where $T = \mR \times M_0$, $(M_0,g_0)$ is a simple $2$-manifold, and $g = c(e \oplus g_0)$ where $c$ is a smooth positive function and $e$ is the Euclidean metric on $\mR$.
\end{definition}

Simple manifolds are defined as follows:

\begin{definition}
A compact manifold $(M_0,g_0)$ with smooth boundary $\partial M_0$ is called \emph{simple} if for each $p$ in $M_0$ the map $\exp_p$ is a diffeomorphism from a closed neighborhood of $0$ in $T_p M_0$ onto $M_0$, and if $\partial M_0$ is strictly convex (meaning that the second fundamental form of $\partial M_0$ is positive definite).
\end{definition}

Admissible manifolds include compact submanifolds of Euclidean space, hyperbolic space, and $S^3$ minus a point, and also sufficiently small submanifolds of conformally flat manifolds. If $M$ is a bounded open set in $\mR^3$ with smooth boundary, equipped with a metric which in some local coordinates $x = (x_1,x')$ has the form 
\begin{equation*}
g(x) = c(x) \left( \begin{array}{cc} 1 & 0 \\ 0 & g_0(x') \end{array} \right),
\end{equation*}
then $(M,g)$ is admissible if $g_0$ is a simple metric in some sufficiently large ball. Also, admissible manifolds are stable under small perturbations of $g_0$. See \cite{DKSaU} for more details.

We will prove the following result, showing that boundary measurements for the Maxwell equations uniquely determine the material parameters in an admissible manifold.

\begin{thm} \label{thm:main}
Let $(M,g)$ be an admissible manifold, and let $(\eps_1, \mu_1)$ and $(\eps_2, \mu_2)$ be two sets of coefficients satisfying \eqref{paramcond1}--\eqref{paramcond3p}. If the admittance maps satisfy $\Lambda_1 = \Lambda_2$, then $\eps_1 \equiv \eps_2$ and $\mu_1 \equiv \mu_2$ in $M$.
\end{thm}

The second result involves Maxwell equations in a bounded domain $\Omega$ in $\mR^3$ with smooth boundary. The coefficients $\eps, \mu$ are assumed to be smooth positive definite symmetric $(1,1)$-tensors. Associated to these tensors are traveltime metrics $g_{\eps}$ and $g_{\mu}$, which are Riemannian metrics in $\Omega$ describing propagation of waves with different polarizations. We shall assume that the velocity of wave propagation is independent of polarization, which amounts to the property that $\eps$ and $\mu$ are in the same conformal class  \cite{KLS}.

The Maxwell equations in $\Omega$ can be written as 
\begin{equation} \label{maxwell_equations_omega}
\left\{ \begin{array}{rl}
\nabla \times \vec{E} &\!\!\!= i\omega \mu \vec{H}, \\
\nabla \times \vec{H} &\!\!\!= -i\omega \eps \vec{E},
\end{array} \right.
\end{equation}
where $\vec{E}$ and $\vec{H}$ are complex vector fields and $\omega > 0$ is a fixed frequency. We consider the electric boundary condition 
\begin{equation} \label{boundary_omega}
\vec{E}_{\text{tan}}|_{\partial \Omega} = \vec{f},
\end{equation}
where $\vec{f}$ is a smooth tangential vector field on $\partial \Omega$ and $\vec{E}_{\text{tan}}|_{\partial \Omega}$ is the tangential part of $\vec{E}|_{\partial \Omega}$. Under the above assumptions, there is a discrete set of resonant frequencies outside which the boundary problem for Maxwell equations has a unique smooth solution $(\vec{E}, \vec{H})$ (see Section \ref{sec:recovery}). The admittance map is given by 
\begin{equation*}
\Lambda: \vec{E}_{\text{tan}}|_{\partial \Omega} \mapsto \vec{H}_{\text{tan}}|_{\partial \Omega}.
\end{equation*}
The next result considers the inverse problem of recovering the electromagnetic parameters from $\Lambda$.

\begin{thm} \label{thm:main2}
Let $\eps_j$ and $\mu_j$ be smooth symmetric positive definite $(1,1)$-tensors on $\closure{\Omega}$, and suppose that $\omega > 0$ is not a resonant frequency for the corresponding boundary problems. Let $\Lambda_j$ be the corresponding admittance maps ($j=1,2$). Assume that there is a fixed admissible metric $g$ in $\closure{\Omega}$ such that $\eps_1$, $\mu_1$, $\eps_2$, and $\mu_2$ are conformal multiples of $g^{-1}$. If the admittance maps satisfy $\Lambda_1 = \Lambda_2$, then $\eps_1 \equiv \eps_2$ and $\mu_1 \equiv \mu_2$ in $\Omega$.
\end{thm}

To our knowledge, Theorems \ref{thm:main} and \ref{thm:main2} are the first positive results on the inverse problem for time-harmonic Maxwell equations in anisotropic settings. For bounded domains in $\mR^3$ where $g$ is the Euclidean metric, Theorems \ref{thm:main} and \ref{thm:main2} were proved in \cite{OPS}.

There has recently been considerable interest in invisibility cloaking \cite{GKLU_survey}, where one looks for anisotropic materials for which uniqueness does not hold. The prescriptions of electromagnetic parameters for cloaking \cite{GKLU} satisfy that $\eps=\mu$. Moreover the parameters are singular, so that one of the eigenvalues is zero at the boundary of the cloaked region. Theorems \ref{thm:main} and \ref{thm:main2} imply that there is no cloaking for materials whose electromagnetic parameters satisfy the given conditions.

Formally, the proofs of Theorems \ref{thm:main} and \ref{thm:main2} follow the Euclidean case. The proof of the uniqueness result in \cite{OPS} was considerably simplified in \cite{OS}, and the simplified proof can be described by the following seven steps:

\begin{enumerate}
\item[1.] 
Reduction of the Maxwell system to a Dirac system, by introducing two auxiliary scalar fields $\Phi$ and $\Psi$. A solution $X$ of the Dirac system gives a solution to the original Maxwell system iff $\Phi = \Psi = 0$.
\item[2.] 
Reduction to a rescaled Dirac system $(P-k+W)Y = 0$, where $Y$ is obtained by rescaling the components of $X$ by $\eps^{1/2}$ and $\mu^{1/2}$.
\item[3.] 
Reduction to the Schr\"odinger equation $(-\Delta - k^2 + Q) Z = 0$, which is possible since $(P-k+W)(P+k-W^t) = -\Delta -k^2 + Q$.
\item[4.] 
Construction of complex geometrical optics solutions to the equation $(-\Delta - k^2 + Q) Z = 0$, which also gives solutions $Y = (P+k-W^t)Z$ to the Dirac system.
\item[5.] 
Construction of solutions to the original Maxwell system. This requires showing that the scalar fields in Step 1 vanish identically, which follows from a uniqueness result for $Z$.
\item[6.] 
Inserting complex geometrical optics solutions in an integral identity, which allows to recover nonlinear differential expressions involving the electromagnetic parameters.
\item[7.] 
An application of the unique continuation principle for a semilinear elliptic system to recover the parameters.
\end{enumerate}

In \cite{OPS_survey}, it was shown that Steps 1 to 3 above can be carried out also for the Maxwell equations on a Riemannian manifold $(M,g)$. However, Step 4 requires complex geometrical optics solutions, and these were only available for the Euclidean metric. Therefore, it was not possible to go further in the non-Euclidean case.

A construction of complex geometrical optics solutions for scalar elliptic equations, valid on admissible Riemannian manifolds $(M,g)$, was given in \cite{DKSaU}. We will combine the ideas in \cite{DKSaU} with the scheme outlined above to prove the uniqueness result for the inverse problem for Maxwell equations on admissible manifolds.

It will turn out that the main technical obstacle is Step 5, which requires a uniqueness result for the complex geometrical optics solutions. In \cite{DKSaU} the construction of solutions is based on Carleman estimates, and there is no concept of uniqueness for the solutions so obtained. In this article we give a new construction of solutions based on direct Fourier arguments. This construction comes with a suitable uniqueness result, which can be used to carry out the proof of the Maxwell result.

The main step in the new construction is a counterpart of the basic norm estimates of Sylvester-Uhlmann \cite{sylvesteruhlmann}. We outline the idea in a simple case. The estimate is valid in $(T,g)$ where $T = \mR \times M_0$ and $g = e \oplus g_0$, but here $(M_0,g_0)$ can be any compact $(n-1)$-dimensional manifold with boundary (no restrictions on the metric). We look for solutions of the equation 
\begin{equation} \label{intro_conjugated_equation}
e^{\tau x_1}(-\Delta_g)(e^{-\tau x_1} u) = f \quad \text{in }T,
\end{equation}
with $\Delta_g$ the Laplace-Beltrami operator in $(T,g)$ and $\tau$ a large parameter.

In the Sylvester-Uhlmann estimates $T =\mR^n$ and $g$ is the Euclidean metric, $f$ is in a weighted $L^2$ space such that $\br{x}^{\delta+1} f \in L^2(\mR^n)$ where $-1 < \delta < 0$, and one obtains a unique solution $u$ with $\br{x}^{\delta} u \in L^2(\mR^n)$. Here 
\begin{equation*}
\br{x} = (1+\abs{x}^2)^{1/2}.
\end{equation*}
In our case we write $x_1$ for the special Euclidean coordinate in $T$, and use Agmon-type weighted spaces
\begin{equation*}
L^2_{\delta}(T) = \{f \in L^2_{\text{loc}}(T) \,;\, \norm{\br{x_1}^{\delta} f}_{L^2(T)} < \infty \}.
\end{equation*}
The Sobolev space $H^s_{\delta}(T)$ is defined via the norm $\norm{u}_{H^s_{\delta}(T)} = \norm{\br{x_1}^{\delta} u}_{H^s(T)}$, and $H^1_{\delta,0}(T)$ is the set $\{ u \in H^1_{\delta}(T) \,;\, u|_{\mR \times \partial M_0} = 0\}$.

The next result is a special case of Proposition \ref{prop:normestimate_uniqueness} (since there is no potential it follows that one may take $\tau_0 = 1$).

\begin{thm}\label{thm:intro_cgo}
Let $\delta > 1/2$. If $\abs{\tau} \geq 1$ is outside a discrete set, then for any $f \in L^2_{\delta}(T)$ there is a unique solution $u \in H^1_{-\delta,0}(T)$ of the equation \eqref{intro_conjugated_equation}. In fact, one has $u \in H^2_{-\delta}(T)$ and 
\begin{equation*}
\norm{u}_{H^s_{-\delta}(T)} \leq C \abs{\tau}^{s-1} \norm{f}_{L^2_{\delta}(T)}, \quad 0 \leq s \leq 2,
\end{equation*}
with $C$ independent of $\tau$.
\end{thm}

In the Sylvester-Uhlmann result, the proof applies the Fourier transform and one obtains uniqueness by fixing decay at infinity. In our case there is a transversal metric in $M_0$, and the Fourier transform or conditions at infinity are not readily available. However, one can ask for decay in the Euclidean variable and Dirichlet boundary values on $\partial M_0$. This makes it possible to use the Fourier transform in $x_1$ and eigenfunction expansions in $M_0$.

The proof of Theorem \ref{thm:intro_cgo} is robust in the sense that one can essentially replace the Laplacian in $M_0$ by any positive operator with a complete set of eigenfunctions. We will need this flexibility in the Maxwell result when proving similar estimates for the Hodge Laplacian on forms. There is also an extra twist in the construction of solutions since one needs a result like Theorem \ref{thm:intro_cgo} which applies to functions $f$ which may not decay (so one is out of the standard Agmon setting), see Sections \ref{sec:norm_estimates} and \ref{sec:norm_estimates_forms} for these more general results.

The construction could also be used to develop constructive methods for certain anisotropic inverse problems. In the Euclidean case, results of this type were given in \cite{nachman_reconstruction} for the 3D conductivity equation and in \cite{OPS} for Maxwell equations.

Earlier work on the inverse problem for the Maxwell system in Euclidean space includes a study of the linearized inverse problem \cite{SIC}, a local uniqueness result \cite{SuU}, and a result for the corresponding inverse scattering problem in the case where $\mu$ is constant \cite{CP}. As mentioned above, the full inverse problem was solved in \cite{OPS}, and in \cite{OS} the proof was simplified and also a reconstruction from measurements based on dipole point sources was given. The paper \cite{OPS_survey} is a survey and also considers the manifold setting. The inverse problem for Maxwell equations in chiral media was considered in \cite{McDowall}. Boundary determination results are given in \cite{JM} and \cite{McDowall_boundary}. Finally, \cite{COS} gives a partial data result for this problem, based on Isakov's method \cite{I}. For results on inverse problems for the Maxwell equations in time domain, we refer to \cite{KLS} and the references therein.

The structure of the paper is as follows. Section \ref{sec:notation} contains notation and identities in Riemannian geometry which will be used throughout the article. The reductions of the Maxwell equations to Dirac and Schr\"odinger equations are given in Section \ref{sec:reductions}. The norm estimates and uniqueness results required for constructing complex geometrical optics solutions are given in Sections \ref{sec:norm_estimates} and \ref{sec:norm_estimates_forms}, and the construction of solutions is taken up in Section \ref{sec:solutions}. In Section \ref{sec:recovery} we prove Theorems \ref{thm:main} and \ref{thm:main2}. There are two appendices, one on the wellposedness theory of boundary value problems for Maxwell, and one including a unique continuation result for principally diagonal systems required for our results.

\subsection*{Acknowledgements}

C.K.~is partly supported by NSF grant DMS0456583, M.S. is supported in part by the Academy of Finland, and G.U.~is partly supported by NSF and a Walker Family Endowed Professorship.

\section{Notation and identities} \label{sec:notation}

We will briefly introduce some basic notation and identities in Riemannian geometry which will be used throughout. We refer to \cite{T1} for these facts.

In this section let $(M,g)$ be a smooth ($=C^{\infty}$) $n$-dimensional Riemannian manifold with or without boundary. All manifolds will be assumed to be oriented. We write $\langle v, w \rangle$ for the $g$-inner product of tangent vectors, and $\abs{v} = \langle v,v \rangle^{1/2}$ for the $g$-norm. If $x = (x_1,\ldots,x_n)$ are local coordinates and $\partial_j$ the corresponding vector fields, we write $g_{jk} = \langle \partial_j, \partial_k \rangle$ for the metric in these coordinates. The determinant of $(g_{jk})$ is denoted by $\abs{g}$, and $(g^{jk})$ is the matrix inverse of $(g_{jk})$.

We shall often do computations in normal coordinates. These are coordinates $x$ defined in a neighborhood of a point $p \in M^{\text{int}}$ such that $x(p) = 0$ and geodesics through $p$ correspond to rays through the origin in the $x$ coordinates. The metric in these coordinates satisfies 
\begin{equation*}
g_{jk}(0) = \delta_{jk}, \quad \partial_l g_{jk}(0) = 0.
\end{equation*}
For points $p \in \partial M$ we will employ boundary normal coordinates, which are coordinates $y = (y',y_n)$ near $p$ so that $y(p) = 0$, $y'$ are normal coordinates on $\partial M$ centered at $p$, and $y_n(q)$ is the geodesic distance from a point $q$ to $\partial M$. The metric has the form 
\begin{equation*}
g(y) = \left( \begin{array}{cc} g_0(y) & 0 \\ 0 & 1 \end{array} \right), \quad g_0(y) = (g_{jk}(y))_{j,k=1}^{n-1},
\end{equation*}
and $g_{jk}(0) = \delta_{jk}$, $\partial_l g_{jk}(0) = 0$. We denote by $\nu$ the $1$-form corresponding to the outer unit normal vector of $\partial M$, so that $\nu = -dy^n$ in boundary normal coordinates.

The Einstein convention of summing over repeated upper and lower indices will be used. We convert vector fields to $1$-forms and vice versa by the musical isomorphisms, which are given by 
\begin{align*}
(X^j \partial_j)^{\flat} = X_k \,dx^k, \quad &X_k = g_{jk} X^j, \\
(\omega_k \,dx^k)^{\sharp} = \omega^j \partial_j, \quad &\omega^j = g^{jk} \omega_k.
\end{align*}
The set of smooth $k$-forms on $M$ is denoted by $\Omega^k M$, and the graded algebra of differential forms is written as 
\begin{equation*}
\Omega M = \oplus_{k=0}^n \Omega^k M.
\end{equation*}
The set of $k$-forms with $L^2$ or $H^s$ coefficients are denoted by $L^2(\Omega^k M)$ and $H^s(\Omega^k M)$, respectively. Here $H^s$ for $s \in \mR$ are the usual Sobolev spaces on $M$. The inner product and norm are extended to forms and more generally tensors on $M$ in the usual way.

Let $d: \Omega^k M \to \Omega^{k+1} M$ be the exterior derivative, and let $*: \Omega^k M \to \Omega^{n-k} M$ be the Hodge star operator. We introduce the sesquilinear inner product on $\Omega^k M$, 
\begin{equation*}
(\eta|\zeta) = \int_M \langle \eta, \bar{\zeta} \rangle \,dV = \int_M \eta \wedge * \bar{\zeta}.
\end{equation*}
Here $dV = *1 = \abs{g}^{1/2} \,dx^1 \cdots \,dx^n$ is the volume form. The codifferential $\delta: \Omega^k M \to \Omega^{k-1} M$ is defined as the formal adjoint of $d$ in the inner product on real valued forms, so that 
\begin{equation*}
(d\eta|\zeta) = (\eta|\delta \zeta), \ \ \text{for } \eta \in \Omega^{k-1} M, \zeta \in \Omega^k M \text{ compactly supported and real}.
\end{equation*}
These operators satisfy the following relations on $k$-forms in $M$:
\begin{equation*}
** = (-1)^{k(n-k)}, \quad \delta = (-1)^{k(n-k)-n+k-1} *d*.
\end{equation*}
If $\xi$ is a $1$-form then the interior product $i_{\xi}$ is the formal adjoint of $\xi \wedge $ in the inner product on real valued forms, and on $k$-forms it has the expression 
\begin{equation*}
i_{\xi} = (-1)^{n(k-1)} *\xi \wedge *.
\end{equation*}
The Hodge Laplacian on $k$-forms is defined by 
\begin{equation*}
-\Delta = (d+\delta)^2 = d\delta + \delta d.
\end{equation*}
It satisfies $\Delta * = * \Delta$.

The Levi-Civita connection, defined on tensors in $M$, is denoted by $\nabla$. We will slightly abuse notation and reserve the expression $\nabla f$ (where $f$ is any function) for the metric gradient of $f$, defined by 
\begin{equation*}
\nabla f = (df)^{\sharp} = g^{jk} \partial_j f \partial_k.
\end{equation*}
The $H^1$ and $H^2$ norms may be expressed invariantly as 
\begin{eqnarray*}
 & \norm{f}_{H^1(M)} = \norm{f}_{L^2(M)} + \norm{\nabla f}_{L^2(M)}, & \\
 & \norm{f}_{H^2(M)} = \norm{f}_{H^1(M)} + \norm{\nabla^2 f}_{L^2(M)}. & 
\end{eqnarray*}
Here of course $\norm{T}_{L^2(M)} = \left( \int_M \abs{T}^2 \,dV \right)^{1/2}$ for a tensor $T$.

For $n=3$, the surface divergence of $f \in H^s(\Omega^1 (\partial M))$ is given by 
\begin{equation*}
\text{Div}(f) = \langle d_{\partial M} f, dS \rangle
\end{equation*}
where $dS$ is the volume form on $\partial M$. A computation in boundary normal coordinates shows that $\text{Div}(f) = -\langle \nu, *du \rangle|_{\partial M}$ where $u \in H^{s+1/2}(\Omega^1 M)$ is any $1$-form with $tu = f$ (here $s > 0$).

Finally, in the case $n=3$, we collect a number of identities which will be useful for computations. Below let $f$ be a smooth function, $\alpha = \alpha_j \,dx^j$ and $\beta = \beta_j \,dx^j$ and $\gamma = \gamma_j \,dx^j$ three $1$-forms, $\eta$ a $k$-form, and $\zeta$ an $l$-form. For the Hodge star one has 
\begin{gather*}
** \eta = \eta, \\
*(\alpha \wedge *\beta) = \langle \alpha, \beta \rangle, \\
*(\alpha \wedge *[\beta \wedge \gamma]) = \langle \alpha, \gamma \rangle \beta - \langle \alpha, \beta \rangle \gamma, \\
\intertext{and the operators $d$ and $\delta$ satisfy }
\quad \delta \eta = (-1)^k *d* \eta, \\
\delta \alpha = -\abs{g}^{-1/2} \partial_j (\abs{g}^{1/2} g^{jk} \alpha_k), \\
d(\eta \wedge \zeta) = d\eta \wedge \zeta + (-1)^k \eta \wedge d\zeta, \\
\delta(f \eta) = f \delta \eta + (-1)^k *df \wedge *\eta, \\
\delta(\alpha \wedge \beta) = (\delta \alpha) \beta - (\delta \beta) \alpha - [\alpha^{\sharp}, \beta^{\sharp}]^{\flat}.
\end{gather*}

\section{Reduction to Schr{\"o}dinger equation} \label{sec:reductions}

In this section we present the reductions of the time-harmonic Maxwell system to Dirac and Schr\"odinger equations, which corresponds to Steps 1 to 3 in the introduction. This mostly follows \cite{OPS_survey} and \cite{OS} although with different notations. We will also give a reduction to the case where the coefficients are constant near the boundary, namely, 
\begin{equation} \label{paramcond3}
\text{$\eps = \eps_0$ and $\mu = \mu_0$ near $\partial M$ for some constants $\eps_0, \mu_0 > 0$}.
\end{equation}

It is well known that the Maxwell system \eqref{maxwell_equations} is not elliptic as it is written. We perform an elliptization by adding the constituent equations, obtained from \eqref{maxwell_equations} by applying $d*$ to both equations:
\begin{equation} \label{constituent_equations}
\left\{ \begin{array}{rl}
d(\mu *H) &\!\!\!= 0, \\
d(\eps *E) &\!\!\!= 0.
\end{array} \right.
\end{equation}
Adding two equations requires adding two unknowns, which will be the scalar fields $\Phi$ and $\Psi$. The choice for how to couple $\Phi$ and $\Psi$ into the larger system obtained from \eqref{maxwell_equations}, \eqref{constituent_equations} was motivated in \cite{OS} by dimensionality arguments. The end result is the following system:
\begin{equation} \label{augmented_equations}
\left\{ \begin{array}{rl}
D*E + D\alpha \wedge *E - \omega \mu *\Phi &\!\!\!= 0, \\
*D\Psi + DE - \omega \mu * H + *D\alpha \wedge \Psi &\!\!\!= 0, \\
D*H + D\beta \wedge *H - \omega \eps *\Psi &\!\!\!= 0, \\
*D\Phi - DH + *D\beta \wedge \Phi - \omega \eps *E &\!\!\!= 0.
\end{array} \right.
\end{equation}
Here we have written $D = \frac{1}{i} d$ and $\alpha = \log\,\eps$, $\beta = \log\,\mu$. We will also write $D^* = -\frac{1}{i} \delta$ for the formal adjoint of $D$ in the sesquilinear inner product on forms.

We wish to express \eqref{augmented_equations} as an equation for the graded differential form $X = \Phi + E + *H + *\Psi$, written in vector notation as 
\begin{equation*}
X = \left( \begin{array}{cc|cc} \Phi & *H & *\Psi & E \end{array} \right)^t.
\end{equation*}
Note that we have grouped the even and odd degree forms together. This will result in a block structure for the equation. Now, taking Hodge star of the first and last equations in \eqref{augmented_equations} results in the system 
\begin{equation} \label{first_dirac_equation}
(P + V)X = 0
\end{equation}
where $P$ and $V$ are given in matrix notation by 
\begin{equation*}
P = \left( \begin{array}{cc|cc} & &  & D^* \\ & & D^* & D \\ \hline  & D & & \\ D & D^* & & \end{array} \right), \ V = \left( \begin{array}{cc|cc} -\omega \mu & & & *D\alpha \wedge * \\ & -\omega \mu & *D\alpha \wedge * & \\ \hline & D\beta \wedge & -\omega \eps & \\ D\beta \wedge & & & -\omega \eps \end{array} \right).
\end{equation*}
This is the first Dirac equation we will use. Note that $P$ is just the self-adjoint Dirac type operator $D + D^*$ on $\Omega M$, and that $(E,H)$ solves the original Maxwell system \eqref{maxwell_equations} iff $X$ solves \eqref{first_dirac_equation} with $\Phi = \Psi = 0$.

For the reduction to a Schr\"odinger equation, it will be convenient to rescale $X$ by 
\begin{equation} \label{rescaling}
X = \left( \begin{array}{c|c} \mu^{-1/2} & \\ \hline & \eps^{-1/2} \end{array} \right) Y,
\end{equation}
where $Y = \left( \begin{array}{cc|cc} Y^0 & Y^2 & Y^3 & Y^1 \end{array} \right)^t$, and $Y^k$ is the $k$-form part of $Y \in \Omega M$. Assuming \eqref{paramcond3} for the moment, a direct computation using the identities in Section \ref{sec:notation} shows that \eqref{first_dirac_equation} is equivalent with the rescaled Dirac equation for $Y$:
\begin{equation} \label{rescaled_dirac_equation}
(P-k+W)Y = 0.
\end{equation}
Here $W$ is the potential, with compact support in $M^{\text{int}}$, given by 
\begin{equation*}
W = -(\kappa-k) + \frac{1}{2} \left( \begin{array}{cc|cc} & & & *D\alpha \wedge * \\ & & *D\alpha \wedge * & -D\alpha \wedge \\ \hline & D\beta \wedge & & \\ D\beta \wedge & *D\beta \wedge * & & \end{array} \right),
\end{equation*}
where $\kappa = \omega (\eps \mu)^{1/2}$, $k = \omega (\eps_0 \mu_0)^{1/2}$.

We will also need the potential $W^t$, which is the formal transpose of $W$ in the inner product on real valued forms, given by 
\begin{equation*}
W^t = -(\kappa-k) + \frac{1}{2} \left( \begin{array}{cc|cc} & & & *D\beta \wedge * \\ & & *D\beta \wedge * & -D\beta \wedge \\ \hline & D\alpha \wedge & & \\ D\alpha \wedge & *D\alpha \wedge * & & \end{array} \right).
\end{equation*}
The adjoint is $W^* = \overline{W^t}$. The following result contains the Schr\"odinger equations, involving the Hodge Laplacian $-\Delta = d\delta + \delta d$ on $\Omega M$, in a form which will be convenient below.

\begin{lemma} \label{lemma:reduction_schrodinger}
We have 
\begin{align*}
(P-k+W)(P+k-W^t) &= -\Delta -k^2 + Q, \\
(P+k-W^t)(P-k+W) &= -\Delta -k^2 + Q', \\
(P-k+W^*)(P+k-\bar{W}) &= -\Delta - k^2 + \hat{Q},
\end{align*}
where $Q$, $Q'$, and $\hat{Q}$ are smooth potentials with compact support in $M^{\text{int}}$,  
\begin{align*}
Q &= k^2 - \kappa^2 + \frac{1}{2} \left( \begin{array}{cc|cc} \Delta \alpha + \frac{1}{2} \langle d\alpha,d\alpha \rangle & 0 & \bullet & \bullet \\ 0 & \bullet & \bullet & \bullet \\ \hline \bullet & \bullet & \Delta \beta + \frac{1}{2} \langle d\beta,d\beta \rangle & 0 \\ \bullet & \bullet & 0 & \bullet \end{array} \right), \\
Q' &= k^2 - \kappa^2 - \frac{1}{2} \left( \begin{array}{cc|cc} \Delta \beta - \frac{1}{2} \langle d\beta,d\beta \rangle & 0 & 0 & 0 \\ \bullet & \bullet & \bullet & \bullet \\ \hline 0 & 0 & \Delta \alpha - \frac{1}{2} \langle d\alpha,d\alpha \rangle & 0 \\ \bullet & \bullet & \bullet & \bullet \end{array} \right),
\end{align*}
and $\bullet$ denote smooth coefficients.
\end{lemma}
\begin{proof}
We give the proof of the first identity, the other ones being analogous. One has 
\begin{equation*}
(P-k+W)(P+k-W^t) = -\Delta -k^2 + W(P+k) - (P-k)W^t - WW^t.
\end{equation*}
The point is to show that the first order term vanishes. We write $W$ as 
\begin{equation*}
W = -(\kappa-k) + \frac{1}{2} W_0,
\end{equation*}
where $W_0$ acts on a graded form $X = X_+ + X_-$, with $X_+ = X^0 + X^2$ and $X_- = X^1 + X^3$, by 
\begin{equation*}
W_0 X = (-D\alpha \wedge + i_{D\alpha})X_- + (D\beta \wedge - i_{D\beta})X_+.
\end{equation*}

We will use the identities in Section \ref{sec:notation}. If $u$ is a $0$-form then 
\begin{align*}
 & (W_0 P - P W_0^t)u = \frac{1}{i}(W_0 du - (d-\delta)(u D\alpha)) \\
 &= - (-d\alpha \wedge du + *d\alpha \wedge *du - du \wedge d\alpha + u (\delta d\alpha) - *du \wedge *d\alpha) \\
 &= (\Delta \alpha)u.
\end{align*}
If $u$ is a $1$-form we have 
\begin{align*}
 & (W_0 P - P W_0^t)u = \frac{1}{i}(W_0 (du - \delta u) - (d-\delta)(-D\beta \wedge u + \langle D\beta, u \rangle)) \\
 &= -(d\beta \wedge du + *d\beta \wedge *du - (\delta u)d\beta - d\beta \wedge du - d \langle d\beta, u \rangle - \delta(d\beta \wedge u)).
\end{align*}
The identity $\delta(d\beta \wedge u) = (-\Delta \beta)u - (\delta u)d\beta - [\nabla \beta, u^{\sharp}]^{\flat}$ and a computation in normal coordinates implies that 
\begin{equation*}
(W_0 P - P W_0^t)u = 2(\nabla^2 \beta) u - (\Delta \beta) u.
\end{equation*}
Here $(\nabla^2 \beta) u$ denotes the $1$-form corresponding to the vector field $(\nabla^2 \beta) u^{\sharp}$. The computation for $2$-forms and $3$-forms can be reduced to the previous cases by noting that if $u$ is a $k$-form, then 
\begin{equation*}
(d-\delta)*u = (-1)^k *(d+\delta)u, \quad (\eta \wedge - i_{\eta}) *u = (-1)^{k-1} *(\eta \wedge + i_{\eta})u.
\end{equation*}
Thus, if $*u$ is a $2$-form then 
\begin{align*}
 &(W_0 P - P W_0^t) *u \\
 &= *[d\alpha \wedge du + (\delta u)d\alpha - *d\alpha \wedge *du - d\alpha \wedge du + d\langle d\alpha,u\rangle + \delta(d\alpha \wedge u)] \\
 &= *[2(\nabla^2 \alpha) u - (\Delta \alpha) u].
\end{align*}
Similarly, if $*u$ is a $3$-form then 
\begin{equation*}
(W_0 P - P W_0^t) *u = (\Delta \beta) *u.
\end{equation*}

We have $P (fu) - f Pu = (Df \wedge + i_{Df}) u$ for a function $f$, so 
\begin{equation*}
W P - P W^t = \frac{1}{2} (W_0 P - P W_0^t) + D\kappa \wedge + i_{D\kappa}.
\end{equation*}
This shows that $(P-k+W)(P+k-W^t) = -\Delta-k^2 + Q$ where $Q$ is an operator of order $0$. Since 
\begin{equation*}
k(W+W^t) - WW^t = k^2-\kappa^2 + \frac{1}{2} \kappa(W_0+W_0^t) - \frac{1}{4} W_0 W_0^t
\end{equation*}
where $W_0 W_0^t X = -[(D\alpha \wedge - i_{D \alpha})^2 X_+ + (D\beta \wedge - i_{D \beta})^2 X_-]$, and since one has $(\xi \wedge - i_{\xi})^2 u = -\langle \xi, \xi \rangle u$ for any $k$-form $u$, we obtain the required expression for $Q$.
\end{proof}

The preceding arguments show how to reduce the original Maxwell system to Dirac and Schr\"odinger equations. In the next lemma, which is similar to \cite[p.~1135]{OS}, we give a reduction on the level of boundary measurements: if the admittance maps for two Maxwell systems coincide, then one has an integral identity involving the potentials $Q_j$ and solutions of the Schr\"odinger and Dirac systems. Note that $Z_1$ has to be related to a solution for Maxwell, but $Y_2$ need not be. This flexibility in the choice of $Y_2$ will simplify the recovery of coefficients.

\begin{lemma} \label{lemma:integral_identity}
Let $(\eps_1,\mu_1)$ and $(\eps_2,\mu_2)$ be two sets of coefficients satisfying \eqref{paramcond1}--\eqref{paramcond3p}, and assume that $\Lambda_1 = \Lambda_2$. After replacing $(M,g)$ by a larger manifold (which is admissible if $(M,g)$ is), one may assume that 
\begin{equation} \label{paramcond5}
\eps_1 = \eps_2 = \eps_0 \text{ and } \mu_1 = \mu_2 = \mu_0 \text{ near $\partial M$ for constants $\eps_0, \mu_0 > 0$,}
\end{equation}
and one has the identity 
\begin{equation} \label{int_id}
((Q_1-Q_2)Z_1|Y_2) = 0
\end{equation}
for any smooth graded forms $Z_j, Y_j$ satisfying the following properties:
\begin{eqnarray*}
 & (P-k+W_1)Y_1 = 0, \quad Y_1 = (P+k-W_1^t)Z_1, & \\
 & (P-k+W_2^*)Y_2 = 0, & \\
 & Y_1^0 = Y_1^3 = 0. &  
\end{eqnarray*}
\end{lemma}

The reduction to the case where \eqref{paramcond5} holds is a consequence of the next boundary determination result.

\begin{thm} \label{thm:boundary_determination}
Let $(M,g)$ be a compact $3$-manifold with smooth boundary, and let $\eps$ and $\mu$ satisfy \eqref{paramcond1}--\eqref{paramcond3p}. Given a point on $\partial M$, the admittance map $\Lambda$ uniquely determines the Taylor series of $\eps$ and $\mu$ at that point in boundary normal coordinates.
\end{thm}
\begin{proof}
This result was proved in \cite{JM}, \cite{McDowall_boundary} in the case where $M$ is a smooth domain in $\mR^3$ and $g$ is the Euclidean metric (the result was for complex $\eps$ and real $\mu$, but the same proof works also for complex $\mu$). The argument proceeds by showing that the admittance map is a pseudodifferential operator on $\partial M$, and by computing the symbol of $\Lambda$ in boundary normal coordinates at a fixed point $p \in \partial M$. One then proves, by looking at the difference of two admittance maps, that the Taylor series of $\eps$ and $\mu$ are uniquely determined at $p$.

Fortunately, if $(M,g)$ is a Riemannian manifold with boundary, the form of the metric in boundary normal coordinates is exactly the same as in the Euclidean case. This means that the arguments of \cite{JM}, \cite{McDowall_boundary}, which were given in boundary normal coordinates of a Euclidean domain, carry over without changes to establish Theorem \ref{thm:boundary_determination} for any Riemannian manifold $(M,g)$.
\end{proof}

\begin{proof}[Proof of Lemma \ref{lemma:integral_identity}]
We first establish \eqref{int_id} in the original manifold $M$. Note that if \eqref{paramcond5} is not satisfied, we may formally take $k=0$ in the preceding arguments and then all conclusions remain valid except that $W_j$ and $Q_j$ may not be compactly supported in $M^{\text{int}}$.

Let $Z_j$ and $Y_j$ be as described, and let $X_1$ be the solution to \eqref{first_dirac_equation}, with potential $V_1$, corresponding to $Y_1$ as in \eqref{rescaling}. Since $\Lambda_1 = \Lambda_2$, we can find a solution $\tilde{X}_2$ of $(P+V_2) \tilde{X}_2 = 0$ with $t\tilde{H}_2 = tH_1$ and $t\tilde{E}_2 = tE_1$ on $\partial M$. Here we write 
\begin{equation*}
X_1 = \left(\begin{array}{cc|cc} 0 & *H_1 & 0 & E_1 \end{array}\right)^t, \quad \tilde{X}_2 = \left(\begin{array}{cc|cc} 0 & *\tilde{H}_2 & 0 & \tilde{E}_2 \end{array}\right)^t.
\end{equation*}
If $\tilde{Y}_2$ is the solution to $(P-k+W_2)\tilde{Y}_2 = 0$ corresponding to $\tilde{X}_2$ as in \eqref{rescaling}, then $t(Y_1-\tilde{Y}_2) = 0$ since $\eps_1 = \eps_2$ and $\mu_1 = \mu_2$ on $\partial M$ by Theorem \ref{thm:boundary_determination}.

We wish to argue that 
\begin{equation} \label{maxwell_boundary_claims}
\nu \wedge (Y_1-\tilde{Y}_2) = 0, \quad i_{\nu} (Y_1-\tilde{Y}_2) = 0 \text{ on } \partial M.
\end{equation}
The first part is immediate since $\nu \wedge \eta = 0$ on $\partial M$ iff $t\eta = 0$. For the second part we use the surface divergence. The fact that $X_1$ and $\tilde{X}_2$ solve the Maxwell equations, together with Theorem \ref{thm:boundary_determination}, implies that 
\begin{align*}
\langle \nu, H_1-\tilde{H}_2 \rangle|_{\partial M} &= \frac{1}{i\omega} \langle \nu, \mu_1^{-1} *dE_1 - \mu_2^{-1} *d\tilde{E}_2 \rangle|_{\partial M} \\
 &= -\frac{1}{i\omega \mu_1} \text{Div}(t(E_1-\tilde{E}_2)) = 0.
\end{align*}
A similar result is true for $E_1-\tilde{E}_2$. This proves \eqref{maxwell_boundary_claims} since we have $*\nu \wedge *\eta = i_{\nu} \eta = \langle \nu,\eta \rangle$ for any $1$-form $\eta$.

Let us next prove that 
\begin{equation*}
((W_1-W_2)Y_1|Y_2) = 0.
\end{equation*}
We have 
\begin{align*}
 & ((W_1-W_2)Y_1|Y_2) = (W_1 Y_1|Y_2) - (Y_1|W_2^* Y_2) \\
 &= (Y_1|(P-k)Y_2) - ((P-k)Y_1|Y_2) \\
 &= (Y_1|(P-k) Y_2) - ((P-k)(Y_1-\tilde{Y}_2)|Y_2) - ((P-k)\tilde{Y}_2|Y_2) \\
 &= (Y_1|(P-k) Y_2) - (Y_1-\tilde{Y}_2|(P-k)Y_2) - ((P-k)\tilde{Y}_2|Y_2).
\end{align*}
In the last step, the boundary term arising from the integration by parts vanishes because of  \eqref{maxwell_boundary_claims}. We obtain 
\begin{align*}
((W_1-W_2)Y_1|Y_2) &= (\tilde{Y}_2|(P-k)Y_2) - ((P-k)\tilde{Y}_2|Y_2) \\
 &= -(\tilde{Y}_2|W_2^* Y_2) + (W_2 \tilde{Y}_2|Y_2) = 0.
\end{align*}

Now \eqref{int_id} will follow if we can prove that 
\begin{equation*}
((W_1-W_2)Y_1|Y_2) = ((Q_1-Q_2)Z_1|Y_2).
\end{equation*}
To show this, we recall that $Q_j$ in Lemma \ref{lemma:reduction_schrodinger} has the form 
\begin{equation*}
Q_j = W_j(P+k) - (P-k)W_j^t - W_j W_j^t.
\end{equation*}
Then 
\begin{align*}
 &((W_1-W_2)Y_1|Y_2) = ((W_1-W_2)(P+k-W_1^t)Z_1|Y_2) \\
 &= ((Q_1 + (P-k)W_1^t) Z_1|Y_2) - ((Q_2 + (P-k)W_2^t + W_2(W_2^t-W_1^t))Z_1|Y_2) \\
 &= ((Q_1-Q_2)Z_1|Y_2) + (W_1^t Z_1|(P-k)Y_2) - (W_2^t Z_1|(P-k)Y_2) \\
 &\quad - ((W_2^t-W_1^t)Z_1|W_2^* Y_2).
\end{align*}
Here, we used that $W_1^t = W_2^t$ on $\partial M$ by Theorem \ref{thm:boundary_determination} so there are no boundary terms. Now \eqref{int_id} follows by using the identity $-W_2^* Y_2 = (P-k)Y_2$ in the last term.

Finally, we show how it is possible to arrange that \eqref{paramcond5} holds. The definition of admissible manifolds allows to find (upon enlarging $(M_0,g_0)$ if necessary) a connected admissible manifold $(\tilde{M},g)$ such that 
\begin{equation*}
M \subset \subset \tilde{M} \subset \subset T.
\end{equation*}
If $\tilde{M}$ is not required to be admissible then any choice $\tilde{M} \supset \supset M$ will do. By the condition $\Lambda_1 = \Lambda_2$ and by Theorem \ref{thm:boundary_determination}, we may extend $\eps_j$ and $\mu_j$ smoothly to $\tilde{M}$ so that $\eps_1 = \eps_2$ and $\mu_1 = \mu_2$ in $\tilde{M} \smallsetminus M$, $\eps_j$ and $\mu_j$ have positive real parts in $\tilde{M}$, and further for some constants $\eps_0, \mu_0$ one has $\eps_1 = \eps_2 = \eps_0$ and $\mu_1 = \mu_2 = \mu_0$ near $\partial \tilde{M}$.

Let now $Z_1$, $Y_1$, $Y_2$ be smooth graded forms in $\tilde{M}$ satisfying the conditions in the statement of the lemma in $\tilde{M}$. Since the restrictions to $M$ satisfy the same conditions in $M$, we have \eqref{int_id} in the set $M$. However, $Q_1 = Q_2$ in $\tilde{M} \smallsetminus M$, so \eqref{int_id} remains valid in $\tilde{M}$. This proves the lemma upon replacing $M$ with $\tilde{M}$.
\end{proof}

\section{Norm estimates and uniqueness} \label{sec:norm_estimates}

In this section let $(M_0,g_0)$ be a compact $(n-1)$-dimensional Riemannian manifold with smooth boundary, without any restrictions on the metric. Consider the cylinder $T = \mR \times M_0$ with metric $g = c(e \oplus g_0)$, where $e$ is the Euclidean metric on $\mR$ and $c$ is any smooth positive function in $T$ satisfying 
\begin{equation*}
c(x_1,x') = 1 \text{ when $\abs{x_1}$ is large}.
\end{equation*}
Here and below, we write $x_1$ for the Euclidean coordinate and $x'$ for coordinates on $M_0$. The Laplace-Beltrami operators on $(T,g)$ and $(M_0,g_0)$ are denoted by $\Delta = \Delta_g$ and $\Delta_{x'} = \Delta_{g_0}$, respectively. We will use the $L^2$ space $L^2(T) = L^2(T, dV_g)$ and the Sobolev spaces $H^s(T)$.

If $\delta \in \mR$, define the weighted norms 
\begin{align*}
\norm{u}_{L^2_{\delta}(T)} &= \norm{\br{x_1}^{\delta} u}_{L^2(T)}, \\
\norm{u}_{H^s_{\delta}(T)} &= \norm{\br{x_1}^{\delta} u}_{H^s(T)}.
\end{align*}
Let $L^2_{\delta}(T)$ and $H^s_{\delta}(T)$ be the corresponding spaces. We also consider the spaces $H^1_{\text{loc}}(T) = \{ u \in L^2_{\text{loc}}(T) \,;\, u \in H^1([-R,R] \times M_0) \text{ for all } R > 0 \}$ and 
\begin{align*}
H^1_{\delta,0}(T) &= \{ u \in H^1_{\delta}(T) \,;\, u|_{\mR \times \partial M_0} = 0\}, \\
H^1_{\text{loc},0}(T) &= \{ u \in H^1_{\text{loc}}(T) \,;\, u|_{\mR \times \partial M_0} = 0\}
\end{align*}

The construction of complex geometrical optics solutions in \cite{DKSaU} and \cite{ksu} is based on limiting Carleman weights. It is shown in \cite{DKSaU} that the function $\varphi(x) = x_1$ is a natural limiting Carleman weight in $(T,g)$. We consider the conjugated Helmholtz operator 
\begin{equation*}
e^{\tau \varphi} (-\Delta - k^2 + q) e^{-\tau \varphi}.
\end{equation*}
The following result gives a norm estimate, corresponding to the Carleman estimate in \cite[Theorem 4.1]{DKSaU}, and a uniqueness result for this operator.

\begin{prop} \label{prop:normestimate_uniqueness}
Let $k \geq 0$ be fixed, let $\delta > 1/2$, and let $q$ be a potential satisfying $\br{x_1}^{2\delta} q \in L^{\infty}(T)$. There exists $\tau_0 \geq 1$ such that whenever 
\begin{equation*}
\abs{\tau} \geq \tau_0 \quad \text{and} \quad \tau^2 + k^2 \notin \text{Spec}(-\Delta_{x'}),
\end{equation*}
then for any $f \in L^2_{\delta}(T)$ there is a unique solution $u \in H^1_{-\delta,0}(T)$ of the equation 
\begin{equation} \label{conjugated_equation_u}
e^{\tau x_1} (-\Delta - k^2 + q) e^{-\tau x_1} u = f \quad \text{in } T.
\end{equation}
Further, $u \in H^2_{-\delta}(T)$, and the solution satisfies the estimates 
\begin{equation*}
\norm{u}_{H^s_{-\delta}(T)} \leq C \abs{\tau}^{s-1} \norm{f}_{L^2_{\delta}(T)}, \quad 0 \leq s \leq 2,
\end{equation*}
with $C$ independent of $\tau$ and $f$.
\end{prop}

For the proof, we first claim that it is enough to consider the case where $c \equiv 1$. To see this, note that if $g = c \tilde{g}$ where $\tilde{g} = e \oplus g_0$, one has the identity 
\begin{multline} \label{laplacebeltrami_conformal_scaling}
c^{\frac{n+2}{4}} (-\Delta_g - k^2 + q) (c^{-\frac{n-2}{4}} v) \\
 = (-\Delta_{\tilde{g}} -k^2 + \left[k^2 (1-c) + cq - c^{\frac{n+2}{4}} \Delta_g(c^{-\frac{n-2}{4}}) \right] )v.
\end{multline}
Consequently, $u$ solves \eqref{conjugated_equation_u} iff $v = c^{\frac{n-2}{4}} u$ solves 
\begin{equation*}
e^{\tau x_1} (-\Delta_{\tilde{g}} -k^2 + \tilde{q}) e^{-\tau x_1} v = c^{\frac{n+2}{4}} f.
\end{equation*}
Here $\tilde{q} = k^2 (1-c) + cq - c^{\frac{n+2}{4}} \Delta_g(c^{-\frac{n-2}{4}})$ is another potential such that $\br{x_1}^{2\delta} \tilde{q} \in L^{\infty}(T)$, since $c = 1$ for $\abs{x_1}$ large. This reduction shows that Proposition \ref{prop:normestimate_uniqueness} will follow from the special case where $c \equiv 1$.

Thus, we assume that $c \equiv 1$ and initially also $q \equiv 0$. Then $g$ has the form 
\begin{equation*}
g(x) = \left( \begin{array}{cc} 1 & 0 \\ 0 & g_0(x') \end{array} \right),
\end{equation*}
and the equation \eqref{conjugated_equation_u} may be written as 
\begin{equation} \label{conjugated_equation_second}
(-\partial_1^2 + 2\tau \partial_1 - \tau^2 - k^2 - \Delta_{x'}) u = f.
\end{equation}
We are looking for a solution $u$ with $u|_{\mR \times \partial M_0} = 0$. This motivates the partial eigenfunction expansions along the transversal manifold: 
\begin{equation*}
u(x_1,x') = \sum_{l=0}^{\infty} \tilde{u}(x_1,l) \phi_l(x'), \quad f(x_1,x') = \sum_{l=0}^{\infty} \tilde{f}(x_1,l) \phi_l(x')
\end{equation*}
where $\phi_l$ are the eigenfunctions of $-\Delta_{x'}$ on $M_0$, satisfying $-\Delta_{x'} \phi_l = \lambda_l \phi_l$ in $M_0$ and $\phi_l|_{\partial M_0} = 0$.

Inserting the expansions of $u$ and $f$ into \eqref{conjugated_equation_second} results in the equations 
\begin{equation} \label{fouriercoeff_eq}
(-\partial_1^2 + 2\tau \partial_1 - \tau^2 - k^2 + \lambda_l) \tilde{u}(x_1,l) = \tilde{f}(x_1,l).
\end{equation}
These are second order ODE for the partial Fourier coefficients. To solve them, we will use the following simple result on solutions of linear ODE involving Agmon type weights.

\begin{lemma} \label{lemma:factorized_ode}
If $\mu = a + ib$ where $a, b$ are real, $a \neq 0$, consider the equation 
\begin{equation} \label{general_ode}
u' - \mu u = f \quad \text{in } \mR.
\end{equation}
There is a unique solution $u = S_{\mu} f \in \mSp(\mR)$ for any $f \in \mSp(\mR)$. One has $S_{\mu}: L^2_{\delta}(\mR) \to L^2_{\delta}(\mR)$ if $\delta \in \mR$, and also the norm estimates 
\begin{eqnarray*}
\norm{S_{\mu} f}_{L^2_{\delta}(\mR)} &\!\!\! \leq \frac{C}{\abs{a}} \norm{f}_{L^2_{\delta}(\mR)}, & \abs{a} \geq 1 \text{ and } \delta \in \mR, \\
\norm{S_{\mu} f}_{L^2_{-\delta}(\mR)} &\!\!\! \leq C \norm{f}_{L^2_{\delta}(\mR)}, & a \neq 0 \text{ and } \delta > 1/2.
\end{eqnarray*}
The constant $C$ only depends on $\delta$.
\end{lemma}
\begin{proof}
We take Fourier transforms in \eqref{general_ode} and observe that for $f \in \mSp(\mR)$, there is a unique solution $u = S_{\mu} f \in \mSp(\mR)$ given by 
\begin{equation*}
u = \mF^{-1} \left\{ m(\xi) \hat{f}(\xi) \right\}
\end{equation*}
where $m(\xi) = (i\xi - \mu)^{-1}$. The condition $a \neq 0$ implies that $m$ is a smooth function which satisfies 
\begin{equation*}
\norm{m^{(k)}}_{L^{\infty}} \leq k! \abs{a}^{-(k+1)}, \quad k = 0,1,2,\ldots.
\end{equation*}
Thus, for any $\delta \in \mR$ we have $m \hat{f} \in H^{\delta}$ if $\hat{f} \in H^{\delta}$, which implies $S_{\mu} f \in L^2_{\delta}$. If $\abs{a} \geq 1$ then $\norm{m \hat{f}}_{H^{\delta}} \leq C_{\delta} \abs{a}^{-1} \norm{\hat{f}}_{H^{\delta}}$ and $\norm{S_{\mu} f}_{L^2_{\delta}} \leq C_{\delta} \abs{a}^{-1} \norm{f}_{L^2_{\delta}}$.

We now assume $f \in L^2_{\delta}$ for $\delta > 1/2$. If $a > 0$, the solution to \eqref{general_ode} is given by 
\begin{equation*}
S_{\mu}f(x) = -\int_x^{\infty} f(t) e^{-\mu(t-x)} \,dt.
\end{equation*}
This has the estimate 
\begin{align*}
\abs{S_{\mu} f(x)} &\leq \int_x^{\infty} \abs{f(t)} \,dt \leq \left( \int_{x}^{\infty} \br{t}^{-2\delta} \,dt \right)^{1/2} \norm{f}_{L^2_{\delta}} \\
 &\leq C_{\delta} \norm{f}_{L^2_{\delta}}
\end{align*}
since $\delta > 1/2$. Thus one has $\norm{S_{\mu} f}_{L^2_{-\delta}} \leq C_{\delta} \norm{f}_{L^2_{\delta}}$ again since $\delta > 1/2$. A similar argument gives the result if $a < 0$.
\end{proof}

\begin{proof}[Proof of Proposition \ref{prop:normestimate_uniqueness}]
As argued above, we may assume $c \equiv 1$. Let us also first take $q \equiv 0$. Then we are looking for solutions to the equation \eqref{conjugated_equation_second}.

Let $0 < \lambda_1 \leq \lambda_2 \leq \ldots$ be the Dirichlet eigenvalues on $-\Delta_{x'}$ in $M_0$, and let $\phi_l \in H^1_0(M_0)$ be the corresponding eigenfunctions normalized so that $\{\phi_l\}_{l=1}^{\infty}$ is an orthonormal basis for $L^2(M_0)$. If $u(x_1,\,\cdot\,) \in L^2(M_0)$ we write 
\begin{equation*}
\tilde{u}(x_1,l) = \int_{M_0} u(x_1,\,\cdot\,) \phi_l \,dV_{g_0}.
\end{equation*}

For uniqueness, let $u \in H^1_{\text{loc},0}(T)$ be a solution of \eqref{conjugated_equation_second} with $f = 0$. This means that for all $\psi \in C^{\infty}_c(T^{\text{int}})$, 
\begin{equation*}
\int_T u (-\partial_1^2 - 2\tau \partial_1 - \tau^2 - k^2 - \Delta_{x'}) \psi \,dV = 0.
\end{equation*}
We choose $\psi(x_1,x') = \chi(x_1) \phi_{lj}(x')$ where $\chi \in C^{\infty}_c(\mR)$ and $\phi_{lj} \in C^{\infty}_c(M_0^{\text{int}})$ with $\phi_{lj} \to \phi_l$ in $H^1(M_0)$ as $j \to \infty$. Since $T = \mR \times M_0$, we have 
\begin{multline*}
\int_{\mR} \left( \int_{M_0} u(x_1,\,\cdot\,) \phi_{lj} \,dV_{g_0} \right) (-\partial_1^2 - 2\tau \partial_1 - \tau^2 - k^2) \chi(x_1) \,dx_1 \\
+ \int_{\mR} \left( \int_{M_0} u(x_1,\,\cdot\,)(-\Delta_{x'} \phi_{lj}) \,dV_{g_0} \right) \chi(x_1) \,dx_1 = 0.
\end{multline*}
One has $-\Delta_{x'} \phi_{lj} \to \lambda_l \phi_l$ in $H^{-1}(M_0)$ as $j \to \infty$. Since $u(x_1,\,\cdot\,) \in H^1_0(M_0)$ for a.e.~$x_1$, we have the limits as $j \to \infty$
\begin{equation*}
\int_{M_0} u(x_1,\,\cdot\,) \phi_{lj} \,dV_{g_0} \to \tilde{u}(x_1,l), \quad \int_{M_0} u(x_1,\,\cdot\,)(-\Delta_{x'} \phi_{lj}) \,dV_{g_0} \to \lambda_l \tilde{u}(x_1,l),
\end{equation*}
which are valid for a.e.~$x_1$ and for all $l$. Dominated convergence implies 
\begin{equation*}
(-\partial_1^2 + 2\tau \partial_1 - \tau^2 - k^2 + \lambda_l) \tilde{u}(x_1,l) = 0 \quad \text{in } \mR
\end{equation*}
for all $l$. By taking Fourier transforms in the $x_1$ variable we obtain 
\begin{equation*}
(\xi_1^2 + 2i\tau \xi_1 - \tau^2 - k^2 + \lambda_l) \hat{u}(\xi_1,l) = 0,
\end{equation*}
with $\hat{u}$ the Fourier transform of $\tilde{u}$ with respect to $x_1$. The symbol $\xi_1^2 + 2i\tau \xi_1 - \tau^2 - k^2 + \lambda_l$ is never zero because of the condition $\tau^2 + k^2 \notin \text{Spec}(-\Delta_{x'})$. This implies that $u \equiv 0$.

Let us next show existence of solutions to \eqref{conjugated_equation_second}. We consider the case $\tau > 0$, the case with negative $\tau$ being analogous. We start by writing \eqref{fouriercoeff_eq}, where $l$ is fixed, in the form 
\begin{equation} \label{conjugated_equation_third}
-\left[ (\partial_1 - \tau)^2 - (\lambda_l - k^2) \right] \tilde{u} = \tilde{f}.
\end{equation}
This equation can be factored into first order equations, where the factorization will depend on the sign of $\lambda_l - k^2$.

If $\lambda_l \geq k^2$, then \eqref{conjugated_equation_third} can be written in the form 
\begin{equation*}
-(\partial_1 - \tau + \sqrt{\lambda_l - k^2}) (\partial_1 - \tau - \sqrt{\lambda_l - k^2} ) \tilde{u} = \tilde{f}.
\end{equation*}
Since $\tilde{f} \in L^2_{\delta}$ and $\tau \neq \sqrt{\lambda_l - k^2}$ by assumption, Lemma \ref{lemma:factorized_ode} implies that there is a solution 
\begin{equation*}
\tilde{u}(\,\cdot\,,l) = -S_{\tau+\sqrt{\lambda_l-k^2}} S_{\tau-\sqrt{\lambda_l-k^2}} \tilde{f}(\,\cdot\,,l).
\end{equation*}
On the other hand, if $\lambda_l < k^2$, then \eqref{conjugated_equation_third} takes the form 
\begin{equation*}
-(\partial_1 - \tau + i\sqrt{k^2 - \lambda_l}) (\partial_1 - \tau - i\sqrt{k^2 - \lambda_l} ) \tilde{u} = \tilde{f}
\end{equation*}
and since $\tau \geq 1$ one has a solution 
\begin{equation*}
\tilde{u}(\,\cdot\,,l) = -S_{\tau+i\sqrt{k^2-\lambda_l}} S_{\tau-i\sqrt{k^2-\lambda_l}} \tilde{f}(\,\cdot\,,l).
\end{equation*}
Lemma \ref{lemma:factorized_ode}, with the trivial estimate $\norm{v}_{L^2_{-\delta}} \leq \norm{v}_{L^2_{\delta}}$, implies that 
\begin{equation*}
\norm{\tilde{u}(\,\cdot\,,l)}_{L^2_{-\delta}} \leq \left\{ \begin{array}{cl}
C \tau^{-2} \norm{\tilde{f}(\,\cdot\,,l)}_{L^2_{\delta}}, & \lambda_l < k^2, \\
C (\tau+\sqrt{\lambda_l-k^2})^{-1} \norm{\tilde{f}(\,\cdot\,,l)}_{L^2_{\delta}}, & \lambda_l \geq k^2, \\ 
C (\lambda_l-k^2)^{-1} \norm{\tilde{f}(\,\cdot\,,l)}_{L^2_{\delta}}, & \lambda_l > k^2 + 4\tau^2. \end{array} \right.
\end{equation*}
At this point $C$ only depends on $\delta$.

We write, for $N \geq 1$, 
\begin{equation} \label{udefinition_utilde}
u_N(x_1,x') = \sum_{l=1}^N \tilde{u}(x_1,l) \phi_l(x').
\end{equation}
The objective is to show that as $N \to \infty$, $u_N$ converges in $H^2_{-\delta}(T)$ to a function $u$ with $u|_{\mR \times \partial M_0} = 0$ and $\norm{u}_{H^s_{-\delta}(T)} \leq C \tau^{s-1} \norm{f}_{L^2_{\delta}(T)}$. If these properties hold, then since $\tilde{u}$ satisfies \eqref{fouriercoeff_eq} one has for $u_N$ 
\begin{equation*}
(-\partial_1^2 + 2\tau \partial_1 - \tau^2 - k^2 - \Delta_{x'}) u_N = \sum_{l=1}^N \tilde{f}(x_1,l) \phi_l(x').
\end{equation*}
Consequently, $u$ will be the required solution of \eqref{conjugated_equation_second}.

Assume that $\tau \geq \tau_0 \geq k$. If $0 \leq s \leq 2$, the estimates for $\tilde{u}$ show that 
\begin{multline} \label{utilde_s_estimate}
\sum_{l=1}^{\infty} \lambda_l^s \norm{\tilde{u}(\,\cdot\,,l)}_{L^2_{-\delta}}^2 \leq C k^{2s} \tau^{-4} \sum_{\lambda_l < k^2} \norm{\tilde{f}(\,\cdot\,,l)}_{L^2_{\delta}}^2 \\
 + C \tau^{2(s-1)} \sum_{k^2 \leq \lambda_l \leq 5\tau^2} \norm{\tilde{f}(\,\cdot\,,l)}_{L^2_{\delta}}^2 + C \sum_{\lambda_l > 5\tau^2} \lambda_l^s (\lambda_l-k^2)^{-2} \norm{\tilde{f}(\,\cdot\,,l)}_{L^2_{\delta}}^2 \\
\leq C \tau^{2(s-1)} \sum_{l=1}^{\infty} \norm{\tilde{f}(\,\cdot\,,l)}_{L^2_{\delta}}^2.
\end{multline}
The last sum on the right converges since 
\begin{equation*}
\norm{f}_{L^2_{\delta}(T)}^2 = \int_{\mR} \br{x_1}^{2\delta} \sum_{l=1}^{\infty} \abs{\tilde{f}(x_1,l)}^2 \,dx_1 = \sum_{l=1}^{\infty} \norm{\tilde{f}(\,\cdot\,,l)}_{L^2_{\delta}}^2.
\end{equation*}
Then the limit $u = \lim_{N \to \infty} u_N$ exists in $L^2_{-\delta}(T)$ because of the estimate \eqref{utilde_s_estimate} with $s=0$. One also obtains the estimate 
\begin{equation*}
\norm{u}_{L^2_{-\delta}(T)} \leq C \tau^{-1} \norm{f}_{L^2_{\delta}(T)}.
\end{equation*}

For the first order derivatives, note that for fixed $l$ 
\begin{equation*}
\partial_1 \tilde{u} = \left\{ \begin{array}{cl} (\tau + \sqrt{\lambda_l-k^2}) \tilde{u} - S_{\tau-\sqrt{\lambda_l-k^2}} \tilde{f}, & \lambda_l \geq k^2, \\
(\tau + i\sqrt{k^2-\lambda_l}) \tilde{u} - S_{\tau-i\sqrt{k^2-\lambda_l}} \tilde{f}, & \lambda_l < k^2, \end{array} \right.
\end{equation*}
so that $\norm{\partial_1 \tilde{u}(\,\cdot\,,l)}_{L^2_{-\delta}} \leq C \norm{\tilde{f}(\,\cdot\,,l)}_{L^2_{\delta}}$ for $\tau$ large. Also, if $\nabla_{x'}$ is the metric gradient on $M_0$, then 
\begin{align*}
\norm{\nabla_{x'} u_N}_{L^2_{-\delta}(T)}^2 &= \int_{\mR} \br{x_1}^{-2\delta} (\nabla_{x'} u_N|\nabla_{x'} u_N)_{M_0} \,dx_1 \\
 &= \int_{\mR} \br{x_1}^{-2\delta} (-\Delta_{x'} u_N|u_N)_{M_0} \,dx_1 \\
 &= \int_{\mR} \sum_{l=1}^N \br{x_1}^{-2\delta} \lambda_l \abs{\tilde{u}(x_1,l)}^2 \,dx_1.
\end{align*}
The estimate \eqref{utilde_s_estimate} with $s = 1$ shows that $u_N$ converges to $u$ in $H^1_{-\delta}(T)$, and $\norm{u}_{H^1_{-\delta}(T)} \leq C \norm{f}_{L^2_{\delta}(T)}$. Since $u_N|_{\mR \times \partial M_0} = 0$, the same is true for $u$.

Regarding the second derivatives, note that $-\Delta_{x'} u_N$ converges in $L^2_{-\delta}(T)$ by \eqref{utilde_s_estimate} with $s = 2$, and then $\norm{-\Delta_{x'} u}_{L^2_{-\delta}(T)} \leq C \tau \norm{f}_{L^2_{\delta}(T)}$. It follows that $-\Delta_{x'} u(x_1,\,\cdot\,) \in L^2(M_0)$ for a.e.~$x_1$. The boundary condition for $u$ and elliptic regularity imply that $u(x_1,\,\cdot\,) \in H^2(M_0)$ for a.e.~$x_1$, and that 
\begin{equation*}
\norm{\nabla_{x'}^2 u(x_1,\,\cdot\,)}_{L^2(M_0)} \leq C \norm{\Delta_{x'} u(x_1,\,\cdot\,)}_{L^2(M_0)}.
\end{equation*}
Thus $\norm{\nabla_{x'}^2 u}_{L^2_{-\delta}(T)} \leq C \tau \norm{f}_{L^2_{\delta}(T)}$. Similar estimates are true for $\nabla_{x'} \partial_1 u$ and for 
\begin{equation*}
\partial_1^2 u = -f + (2\tau \partial_1 - \tau^2 - k^2 - \Delta_{x'}) u.
\end{equation*}
This proves that $\norm{u}_{H^2_{-\delta}(T)} \leq C \tau \norm{f}_{L^2_{\delta}(T)}$ as required.

It remains to consider the case of nonzero $q$. Denote by $G_{\tau}$ the solution operator constructed above for the free case $q \equiv 0$ (we may still assume that $c \equiv 1$), and let $q$ be such that $\br{x_1}^{2\delta} q \in L^{\infty}(T)$. If $u \in H^1_{-\delta,0}(T)$ solves \eqref{conjugated_equation_u} with $f \equiv 0$, then 
\begin{equation*}
e^{\tau x_1} (-\Delta - k^2) e^{-\tau x_1} u = -qu.
\end{equation*}
Since $qu \in L^2_{\delta}(T)$, the uniqueness result for the free case shows that $u = -G_{\tau}(qu)$ and $\norm{u}_{L^2_{-\delta}(T)} \leq C \tau^{-1} \norm{qu}_{L^2_{\delta}(T)}$. If $\tau$ is sufficiently large, this implies $u \equiv 0$. For existence of a solution to \eqref{conjugated_equation_u}, we try $u = G_{\tau} v$ where $v \in L^2_{\delta}(T)$ should satisfy 
\begin{equation*}
(I + qG_{\tau})v = f \quad \text{in } T.
\end{equation*}
If $\tau$ is sufficiently large then $q G_{\tau}$ is an operator on $L^2_{\delta}(T)$ with norm $\leq 1/2$, and consequently one can take $v = (I+qG_{\tau})^{-1} f$ with $\norm{v}_{L^2_{\delta}(T)} \leq 2 \norm{f}_{L^2_{\delta}(T)}$. The required norm estimates follow from the estimates for $G_{\tau}$.
\end{proof}

\begin{remark}
If $q \in L^{\infty}(T)$ is compactly supported, the preceding proof shows that the claims in Proposition \ref{prop:normestimate_uniqueness} remain true if one looks for a unique solution $u \in H^1_{\text{loc},0}(T)$ instead of $u \in H^1_{-\delta,0}(T)$.
\end{remark}

\begin{definition}
We let $G_{\tau} : L^2_{\delta}(T) \to H^2_{-\delta} \cap H^1_{-\delta,0}(T)$ be the solution operator given in Proposition \ref{prop:normestimate_uniqueness} in the case $q \equiv 0$.
\end{definition}

In the construction of solutions, we will need to apply the operator $G_{\tau}$ to functions which may not decay in $x_1$. The proof of Proposition \ref{prop:normestimate_uniqueness} involves operators $S_{\mu}$ in two cases: where $\abs{\re(\mu)}$ is large and where $\re(\mu)$ may be close to $0$. The latter case is problematic since good estimates may not be available if there is no decay in $x_1$. However, we only need to apply $G_{\tau}$ to functions of special form: the behaviour in $x_1$ can be assumed to be like $e^{i\lambda x_1}$ where $\lambda > 0$. The following result will be sufficient for our purposes.

\begin{prop} \label{prop:normestimate_uniqueness_2}
Let $k \geq 0$ be fixed, let $\delta > 1/2$ and $\lambda > 0$, and suppose that $\br{x_1}^{2\delta} q \in L^{\infty}(T)$. There exists $\tau_0 \geq 1$ (independent of $\lambda$) such that whenever 
\begin{equation*}
\abs{\tau} \geq \tau_0 \quad \text{and} \quad \tau^2 + k^2 \notin \text{Spec}(-\Delta_{x'}),
\end{equation*}
then for any $f = f_1 + f_2$ where $f_1 \in L^2_{\delta}(T)$, $f_2 \in L^2_{-\delta}(T)$, and 
\begin{equation} \label{xi1_support_condition}
\text{$\mF_{x_1} f_2(\,\cdot\,,x')$ has support in $\{ \abs{\xi_1} \geq \lambda \}$ for a.e.~$x' \in M_0$,}
\end{equation}
there is a unique solution $u \in H^1_{-\delta,0}(T)$ of the equation 
\begin{equation*}
e^{\tau x_1} (-\Delta - k^2 + q) e^{-\tau x_1} u = f \quad \text{in } T.
\end{equation*}
Further, $u \in H^2_{-\delta}(T)$, and the solution satisfies the estimates 
\begin{equation*}
\norm{u}_{H^s_{-\delta}(T)} \leq C \abs{\tau}^{s-1} \left[ \norm{f_1}_{L^2_{\delta}(T)} + \norm{f_2}_{L^2_{-\delta}(T)} \right], \quad 0 \leq s \leq 2,
\end{equation*}
with $C$ independent of $\tau$ and $f_1$, $f_2$.
\end{prop}
\begin{proof}
We follow the proof of Proposition \ref{prop:normestimate_uniqueness}, and may assume $c \equiv 1$ (this uses that $c^{\frac{n+2}{4}} f = [c^{\frac{n+2}{4}} f_1 + \chi f_2] + f_2$ where $\chi$ is compactly supported in $x_1$, so the term in brackets is in $L^2_{\delta}(T)$) and $\tau > 0$. Uniqueness is proved similarly as in Proposition \ref{prop:normestimate_uniqueness}, and that result also gives existence if $f_2 \equiv 0$. Thus, it is enough to consider the case where $f = f_2$.

Assume first that $q \equiv 0$. Let $\tau \geq \tau_0 \geq k$, and recall that $G_{\tau}$ is defined by 
\begin{align*}
G_{\tau} f(x_1,x') &= - \sum_{\lambda_l < k^2} \left[ S_{\tau+i\sqrt{k^2-\lambda_l}} S_{\tau-i\sqrt{k^2-\lambda_l}} \tilde{f}(\,\cdot\,,l) \right](x_1) \phi_l(x') \\
 &\quad - \sum_{k^2 \leq \lambda_l \leq 5\tau^2} \left[ S_{\tau+\sqrt{\lambda_l-k^2}} S_{\tau-\sqrt{\lambda_l-k^2}} \tilde{f}(\,\cdot\,,l) \right](x_1) \phi_l(x') \\
 &\quad - \sum_{\lambda_l > 5\tau^2} \left[ S_{\tau+\sqrt{\lambda_l-k^2}} S_{\tau-\sqrt{\lambda_l-k^2}} \tilde{f}(\,\cdot\,,l) \right](x_1) \phi_l(x').
\end{align*}
By Lemma \ref{lemma:factorized_ode} the first and third sums satisfy $\norm{\,\cdot\,}_{L^2_r(T)} \leq C \tau^{-2} \norm{f}_{L^2_r(T)}$ for any real number $r$. For the second sum we need to analyze the operator $S_{\mu}$ more carefully. If $\mu \in \mR \smallsetminus \{0\}$, and if $w \in L^2_r(\mR)$ with $\hat{w}(\xi) = 0$ for $\abs{\xi} < \lambda$, we have 
\begin{equation*}
S_{\mu} w(x) = \mF^{-1}\{ m_{\lambda}(\xi) \hat{w}(\xi) \}
\end{equation*}
where $m_{\lambda}(\xi) = \psi(\xi/\lambda) (i\xi-\mu)^{-1}$ and $\psi$ is a fixed smooth function satisfying $\psi = 0$ for $\abs{\xi} \leq 1/2$ and $\psi = 1$ for $\abs{\xi} \geq 1$. Since 
\begin{equation*}
\abs{m_{\lambda}^{(k)}(\xi)} \leq C_k \lambda^{-1-k},
\end{equation*}
we have $\norm{S_{\mu} w}_{L^2_r} \leq C_{r,\lambda} \norm{w}_{L^2_r}$.

Using the assumption \eqref{xi1_support_condition}, the last estimate for $S_{\mu}$, and Lemma \ref{lemma:factorized_ode}, we have for $k^2 \leq \lambda_l \leq 5\tau^2$ that 
\begin{align*}
\norm{S_{\tau+\sqrt{\lambda_l-k^2}} S_{\tau-\sqrt{\lambda_l-k^2}} \tilde{f}(\,\cdot\,,l)}_{L^2_r} &\leq \frac{C_r}{\tau+\sqrt{\lambda_l-k^2}} \norm{S_{\tau-\sqrt{\lambda_l-k^2}} \tilde{f}(\,\cdot\,,l)}_{L^2_r} \\
 &\leq \frac{C_{r,\lambda}}{\tau+\sqrt{\lambda_l-k^2}} \norm{\tilde{f}(\,\cdot\,,l)}_{L^2_r}.
\end{align*}
It follows that $G_{\tau}$ maps $L^2_r(T)$ to $L^2_r(T)$ with norm $\leq C \tau^{-1}$. The proof that $G_{\tau}$ maps into $H^2_r \cap H^1_{r,0}(T)$ with the right norm estimates is similar to the corresponding part in Proposition \ref{prop:normestimate_uniqueness}.

It remains to prove existence when $f = f_2$ and $q$ is a potential with $\br{x_1}^{2\delta} q \in L^{\infty}(T)$. We seek a solution $u = G_{\tau} v$, where $v$ solves 
\begin{equation*}
(I + qG_{\tau})v = f.
\end{equation*}
Here $f$ satisfies the support condition \eqref{xi1_support_condition}, but solving this equation by Neumann series involves multiplication with $q$ which breaks the support condition. However, we obtain a solution $v = f + \tilde{v}$ if $\tilde{v}$ satisfies 
\begin{equation*}
(I+qG_{\tau})\tilde{v} = -qG_{\tau} f.
\end{equation*}
The right hand side is in $L^2_{\delta}(T)$ since $G_{\tau}$ maps $f = f_2$ into $L^2_{-\delta}(T)$. We may then use the estimate in Proposition \ref{prop:normestimate_uniqueness} to show that for large $\tau$ there is a solution $\tilde{v} \in L^2_{\delta}(T)$ with $\norm{\tilde{v}}_{L^2_{\delta}} \leq C \tau^{-1} \norm{f}_{L^2_{-\delta}}$. Thus, we obtain a solution to the original equation having the form 
\begin{equation*}
u = G_{\tau} f + G_{\tau} \tilde{v}.
\end{equation*}
This satisfies $\norm{u}_{H^s_{-\delta}(T)} \leq C \tau^{s-1} \norm{f}_{L^2_{-\delta}(T)}$.
\end{proof}

\begin{remark}
As in the remark after the proof of Proposition \ref{prop:normestimate_uniqueness}, if $q \in L^{\infty}(T)$ is compactly supported, the claims in Proposition \ref{prop:normestimate_uniqueness_2} remain valid if one looks for a unique solution $u \in H^1_{\text{loc},0}(T)$ instead of $u \in H^1_{-\delta,0}(T)$.
\end{remark}

\section{Norm estimates for differential forms} \label{sec:norm_estimates_forms}

The purpose in this section is to prove a counterpart of Proposition \ref{prop:normestimate_uniqueness_2} which applies to the Hodge Laplacian on differential forms. We will assume that $M_0$ and $T$ are as in Section \ref{sec:norm_estimates}. For simplicity, we make the further assumptions that $M_0$ is two dimensional (so that $T$ has dimension $3$) and the conformal factor satisfies $c \equiv 1$. 

Let $\Omega T = \Omega^0 T \oplus \Omega^1 T \oplus \Omega^2 T \oplus \Omega^3 T$ be the graded algebra of differential forms. If $U$ is in $\Omega T$, as in Section \ref{sec:reductions} we use the vector notation 
\begin{equation} \label{Uform}
U = ( \begin{array}{cc|cc} R^0 & *S^1 & *S^0 & R^1 \end{array} )^t
\end{equation}
where $R^j, S^j \in \Omega^j T$ ($j=0,1$).

Let $-\Delta = d \delta + \delta d$ be the Hodge Laplacian on $\Omega T$, and write $-\Delta_{x'} = d_{x'} \delta_{x'} + \delta_{x'} d_{x'}$ for the Hodge Laplacian on $\Omega M_0$. We will sometimes write $-\Delta_{x'}^j$ for $-\Delta_{x'}$ acting on $j$-forms. Similarly to Section \ref{sec:norm_estimates}, the proof of norm estimates for the conjugated Laplacian will require an orthonormal set of eigenvectors for $-\Delta_{x'}$. In the case of $0$-forms, we already used the orthonormal basis $\{\phi_l\}_{l=1}^{\infty}$ of $L^2(M_0)$, where 
\begin{equation*}
-\Delta_{x'} \phi_l = \lambda_l \phi_l \ \ \text{in } M_0, \quad \phi_l = 0 \ \ \text{on } \partial M_0.
\end{equation*}
Here $0 < \lambda_1 \leq \lambda_2 \leq \ldots$ are the eigenvalues of the Laplace-Beltrami operator, counted with multiplicity.

In the case of $1$-forms, one needs to make a choice of boundary conditions to fix the orthonormal basis. In view of the applications to Maxwell equations, the relative boundary conditions (see \cite[Section 5.9]{T1}) will be the right choice: there exists an orthonormal basis $\{\psi_m\}_{m=1}^{\infty}$ of $L^2 (\Omega^1 M_0)$ of real valued forms such that 
\begin{equation} \label{psim_eigenvectors}
-\Delta_{x'} \psi_m = \mu_m \psi_m \ \ \text{in } M_0, \quad t \psi_m = t(\delta_{x'} \psi_m) = 0,
\end{equation}
with $0 \leq \mu_1 \leq \mu_2 \leq \ldots$ the eigenvalues of $-\Delta_{x'}$ acting on $1$-forms.

We define Sobolev spaces with relative boundary values:
\begin{align*}
H^s_{\text{R}}(\Omega^0 T) &= \{ u \in H^s(\Omega^0 T) \,;\, u|_{\partial T} = 0\} \quad (s > 1/2), \\
H^s_{\text{R}}(\Omega^1 T) &= \{ u \in H^s(\Omega^1 T) \,;\, t u = t(\delta u) = 0 \text{ on } \partial T \} \quad (s > 3/2),
\end{align*}
and 
\begin{align*}
H^s_{\text{R}}(\Omega T) &= \{ u \in H^s(\Omega T) \,;\, \text{$u$ has the form \eqref{Uform} and $R_j, S_j \in H^s_{\text{R}}(\Omega^j T)$} \}.
\end{align*}
We say that $u$ is in $L^2_{\delta}(\Omega T)$ (respectively $H^s_{\delta}(\Omega T)$) if $\br{x_1}^{\delta} u \in L^2(\Omega T)$ (respectively $\br{x_1}^{\delta} u \in H^s(\Omega T)$). These spaces have the norms 
\begin{align*}
\norm{u}_{L^2_{\delta}(\Omega T)} &= \norm{\br{x_1}^{\delta} u}_{L^2(\Omega T)}, \\
\norm{u}_{H^s_{\delta}(\Omega T)} &= \norm{\br{x_1}^{\delta} u}_{H^s(\Omega T)}.
\end{align*}
Also, $u$ is in $H^s_{\delta,\text{R}}(\Omega T)$ iff $\br{x_1}^{\delta} u \in H^s_{\text{R}}(\Omega T)$.

We may now state the norm estimates and uniqueness result for graded forms.

\begin{prop} \label{prop:normestimate_uniqueness_potential_forms}
Let $k \geq 0$ be fixed, let $\delta > 1/2$ and $\lambda > 0$, and suppose that $Q: L^2_{-\delta}(\Omega T) \to L^2_{\delta}(\Omega T)$ is a bounded linear operator. There exists $\tau_0 \geq 1$ such that whenever 
\begin{equation*}
\abs{\tau} \geq \tau_0 \quad \text{and} \quad \tau^2 + k^2 \notin \text{Spec}(-\Delta_{x'}^0) \cup \text{Spec}(-\Delta_{x'}^1),
\end{equation*}
and whenever $F = F_1 + F_2$ where $F_1 \in L^2_{\delta}(\Omega T)$, $F_2 \in L^2_{-\delta}(\Omega T)$, and $F_2$ is of the form $F_2(x) = w(x_1) \tilde{F}_2$ with $w$ a scalar function and 
\begin{equation*}
\supp(\hat{w}) \subseteq \{ \abs{\xi} \geq \lambda \}, \quad \tilde{F}_2 \in L^{\infty}(\Omega T) \text{ with } \nabla_{\partial_1} \tilde{F}_2 = 0,
\end{equation*}
then there is a unique solution $U \in H^2_{-\delta,\text{R}}(\Omega T)$ of the equation 
\begin{equation*}
e^{\tau x_1} (-\Delta - k^2 + Q) e^{-\tau x_1} U = F \quad \text{in } T.
\end{equation*}
The solution satisfies the estimates 
\begin{equation*}
\norm{U}_{H^s_{-\delta}(\Omega T)} \leq C \abs{\tau}^{s-1} \left[ \norm{F_1}_{L^2_{\delta}(\Omega T)} + \norm{F_2}_{L^2_{-\delta}(\Omega T)} \right], \quad 0 \leq s \leq 2,
\end{equation*}
with $C$ independent of $\tau$ and $F_1$, $F_2$.
\end{prop}
\begin{proof}
Assume $\tau > 0$, and first consider the case where $Q = 0$. If $v$ is a $0$-form on $T$, we have already observed that 
\begin{equation} \label{conjhelmholtz_0form}
e^{\tau x_1} (-\Delta - k^2) e^{-\tau x_1} v = (-\partial_1^2 + 2\tau \partial_1 - \tau^2 - k^2 - \Delta_{x'}) v.
\end{equation}
If $\eta$ is a $1$-form in $T$, we write $\eta = \eta_1 \,dx^1 + \eta'$ where $\eta_1 = \langle \eta, dx^1 \rangle$ and where $\eta' = \eta_2 \,dx^2 + \eta_3 \,dx^3$ is a $1$-form on $M_0$ depending on the parameter $x_1$. A direct computation in normal coordinates, using the identities in Section \ref{sec:notation}, the fact that $g(x_1,x') = \left( \begin{smallmatrix} 1 & 0 \\ 0 & g_0(x') \end{smallmatrix} \right)$, and the identity $\delta \eta = -\partial_1 \eta_1 + \delta_{x'} \eta'$, implies that 
\begin{eqnarray*}
 & -\Delta(\eta_1 \,dx^1) = (-\Delta \eta_1) \,dx^1, & \\
 & -\Delta \eta' = -\Delta_{x'} \eta' - (\partial_1^2 \eta_2) \,dx^2 - (\partial_1^2 \eta_3) \,dx^3. & 
\end{eqnarray*}
Thus, replacing $\eta$ by $e^{-\tau x_1} \eta$ and using that $\nabla_{\partial_1} \eta' = (\partial_1 \eta_2) \,dx^2 + (\partial_1 \eta_3) \,dx^3$, we obtain 
\begin{multline} \label{conjhelmholtz_1form}
e^{\tau x_1} (-\Delta - k^2) e^{-\tau x_1} \eta = \left[ (-\partial_1^2 + 2\tau \partial_1 - \tau^2 - k^2 - \Delta_{x'}) \eta_1 \right] dx^1 + \\
 + (-\nabla_{\partial_1}^2 + 2\tau \nabla_{\partial_1} - \tau^2 - k^2 - \Delta_{x'}) \eta'.
\end{multline}
The formulas \eqref{conjhelmholtz_0form} and \eqref{conjhelmholtz_1form} give explicit expressions for the conjugated Helmholtz operator acting on $0$-forms and $1$-forms. Now $*$ commutes with $e^{\tau x_1} (-\Delta-k^2) e^{-\tau x_1}$ since it commutes with $-\Delta$, so we have a corresponding expression for $e^{\tau x_1} (-\Delta-k^2) e^{-\tau x_1}$ acting on graded forms written as \eqref{Uform}.

Let $U$ be as in \eqref{Uform}, and let $F = ( \begin{array}{cc|cc} F^0 & *G^1 & *G^0 & F^1 \end{array} )^t$. Write $R^1 = R^1_1 \,dx^1 + (R^1)'$, and similarly for $S^1, F^1, G^1$. The equation 
\begin{equation*}
e^{\tau x_1} (-\Delta - k^2) e^{-\tau x_1} U = F \quad \text{in } T
\end{equation*}
can be written in terms of components as 
\begin{eqnarray}
 & (-\partial_1^2 + 2\tau \partial_1 - \tau^2 - k^2 - \Delta_{x'}) \left\{ \begin{array}{l} R^0 \\ S^0 \end{array} \right. = \left\{ \begin{array}{l} F^0 \\ G^0 \end{array} \right., & \label{estimateform_eq1} \\
 & (-\partial_1^2 + 2\tau \partial_1 - \tau^2 - k^2 - \Delta_{x'}) \left\{ \begin{array}{l} R^1_1 \\ S^1_1 \end{array} \right. = \left\{ \begin{array}{l} F^1_1 \\ G^1_1 \end{array} \right., & \label{estimateform_eq2} \\
 & (-\nabla_{\partial_1}^2 + 2\tau \nabla_{\partial_1} - \tau^2 - k^2 - \Delta_{x'}) \left\{ \begin{array}{l} (R^1)' \\ (S^1)' \end{array} \right. = \left\{ \begin{array}{l} (F^1)' \\ (G^1)' \end{array} \right.. & \label{estimateform_eq3}
\end{eqnarray}
The existence and uniqueness of solutions to \eqref{estimateform_eq1} and \eqref{estimateform_eq2} follows from Proposition \ref{prop:normestimate_uniqueness_2}.

For the last two equations, we express $(F^1)'$ and $(G^1)'$ in terms of the eigenvectors \eqref{psim_eigenvectors} as 
\begin{align*}
(F^1)'(x_1,x') &= \sum_{m=1}^{\infty} \widetilde{(F^1)'}(x_1,m) \psi_m(x'), \\
(G^1)'(x_1,x') &= \sum_{m=1}^{\infty} \widetilde{(G^1)'}(x_1,m) \psi_m(x').
\end{align*}
We look for $(R^1)'$ and $(S^1)'$ in a similar form. Then \eqref{estimateform_eq3} is equivalent with the following equations for the partial Fourier coefficients:
\begin{equation}
(-\partial_1^2 + 2\tau \partial_1 - \tau^2 - k^2 + \mu_m) \left\{ \begin{array}{l} \widetilde{(R^1)'}(x_1,m) \\ \widetilde{(S^1)'}(x_1,m) \end{array} \right. = \left\{ \begin{array}{l} \widetilde{(F^1)'}(x_1,m) \\ \widetilde{(G^1)'}(x_1,m). \end{array} \right. \label{relative_fourier_equations}
\end{equation}
Since $\tau^2 + k^2 \notin \text{Spec}(-\Delta_{x'}^1)$, we may use the method in Propositions \ref{prop:normestimate_uniqueness} and \ref{prop:normestimate_uniqueness_2} to solve \eqref{estimateform_eq3}.

More precisely, if $U \in H^2_{-\delta,\text{R}}(\Omega T)$ and the right hand sides in \eqref{estimateform_eq3} are zero, then the relative boundary conditions imply that $(-\Delta_{x'}(R^1)'|\psi_m) = \mu_m ((R^1)'|\psi_m)$ and one obtains for all $m$ 
\begin{equation*}
(-\partial_1^2 + 2\tau \partial_1 - \tau^2 - k^2 + \mu_m) ((R^1)'|\psi_m) = 0.
\end{equation*}
Thus $((R^1)'|\psi_m) = 0$ for all $m$, showing that $(R^1)' = 0$. The same argument applies to $(S^1)'$. Existence follows by solving \eqref{relative_fourier_equations} as in Propositions \ref{prop:normestimate_uniqueness} and \ref{prop:normestimate_uniqueness_2} and by writing $(R^1)'$ and $(S^1)'$ in terms of the Fourier coefficients.

We have given the proof in the case $Q \equiv 0$. However, the case where $Q$ is bounded operator $L^2_{-\delta}(T) \to L^2_{\delta}(T)$ is completely analogous to the corresponding parts of Propositions \ref{prop:normestimate_uniqueness} and \ref{prop:normestimate_uniqueness_2}.
\end{proof}

\section{Construction of solutions} \label{sec:solutions}

In this section we present a construction of complex geometrical optics solutions to the various Schr\"odinger, Dirac, and Maxwell equations which were introduced in Section \ref{sec:reductions}.

Let $(M_0,g_0)$ be a simple $2$-manifold, and let $(\tilde{M}_0,g_0)$ be another simple $2$-manifold with $M_0 \subseteq \tilde{M}_0^{\text{int}}$. Write $\tilde{T} = \mR \times \tilde{M}_0$, $T = \mR \times M_0$, and $T^{\text{int}} = \mR \times M_0^{\text{int}}$ for the various cylinders. Assume that $\tilde{T}$ is equipped with the Riemannian metric 
\begin{equation} \label{metric_form}
g(x_1,x') = \begin{pmatrix} 1 & 0 \\ 0 & g_0(x') \end{pmatrix}.
\end{equation}
The following result provides the solutions which will be used in the integral identity of Lemma \ref{lemma:integral_identity} to recover the coefficients. Part (a) corresponds to a solution for the Maxwell system, and part (b) gives a solution to the Dirac system. We write $-\Delta_{x'}$ for the Hodge Laplacian in $(M_0,g_0)$.

\begin{thm} \label{thm:cgo_main}
Let $(M,g) \subset \subset (T,g)$ be a compact manifold with boundary. Assume that $\eps$ and $\mu$ are coefficients in $M$ satisfying \eqref{paramcond1}, \eqref{paramcond2},  \eqref{paramcond3}. Let $p$ be a point in $\tilde{M}_0 \smallsetminus M_0$, and let $(r,\theta)$ be polar normal coordinates in $\tilde{M}_0$ with center $p$.

There exists $\tau_0 \geq 1$ such that for any $\tau$ with 
\begin{equation*}
\abs{\tau} \geq \tau_0 \quad \text{and} \quad \tau^2 + k^2 \notin \text{Spec}(-\Delta_{x'}^0) \cup \text{Spec}(-\Delta_{x'}^1),
\end{equation*}
and for any constants $s_0, t_0 \in \mR$, the following statements hold:

\medskip

\noindent (a) For any constant $\lambda > 0$ and for any smooth function $\chi = \chi(\theta)$, there exists a solution to $(-\Delta-k^2+Q)Z = 0$ in $M$ such that one has in $M$ 
\begin{eqnarray*}
 & (P-k+W)Y = 0, \quad Y = (P+k-W^t)Z, & \\
 & Y^0 = Y^3 = 0, &  
\end{eqnarray*}
where $Z$ has the form 
\begin{equation} \label{zform_parta}
Z = e^{-\tau(x_1+ir)} \left[ \abs{g}^{-1/4} e^{i\lambda(x_1+ir)} \chi(\theta) \left( \begin{array}{c} s_0 \\ 0 \\ \hline t_0 *1 \\ 0 \end{array} \right) + R \right]
\end{equation}
and $\norm{R}_{L^2(\Omega M)} \leq C \abs{\tau}^{-1}$ where $C$ is independent of $\tau$.

\medskip

\noindent (b) There exists a solution to $(P-k+W^*)Y = 0$ in $M$ of the form 
\begin{equation} \label{yform_partb}
Y = e^{-\tau(x_1+ir)} \left[ \abs{g}^{-1/4} \left( \begin{array}{c} s_0 \\ -is_0 \,dx^1 \wedge dr \\ \hline t_0 *1 \\ it_0 *dx^1 \wedge dr \end{array} \right) + R \right]
\end{equation}
where $\norm{R}_{L^2(\Omega M)} \leq C \abs{\tau}^{-1}$ with $C$ independent of $\tau$.
\end{thm}

To prove this, we begin by considering the Schr\"odinger equation in $(T, g)$, where $k > 0$ is a constant and $Q$ is a smooth potential with compact support in $T^{\text{int}}$. Following \cite[Section 5]{DKSaU}, we wish to construct a solution to 
\begin{equation} \label{schrodinger_t}
(-\Delta - k^2 + Q)Z = 0 \quad \text{in } T
\end{equation}
by using a WKB ansatz with complex phase function, having the form 
\begin{equation} \label{z_form}
Z = e^{-\tau \rho}(A + R).
\end{equation}
Here $\rho = \varphi + i\psi$ is a complex weight, where $\varphi(x) = x_1$ is the limiting Carleman weight. Also, $\tau > 0$ is a large parameter, $A \in \Omega T$ is an amplitude, and $R \in H^2_{-\delta}(\Omega T)$ is a correction term where $\delta > 1/2$.

Introduce the conjugated operators on $\Omega^l T$, 
\begin{align*}
d_{\tau} &= e^{\tau \rho} d e^{-\tau \rho} = d - \tau \,d\rho \wedge, \\
\delta_{\tau} &= e^{\tau \rho} \delta e^{-\tau \rho} = \delta + (-1)^{l+1} \tau * d\rho \wedge *.
\end{align*}
The conjugated Hodge Laplacian is then given by 
\begin{equation*}
-\Delta_{\tau} = e^{\tau \rho} (-\Delta) e^{-\tau \rho} = d_{\tau} \delta_{\tau} + \delta_{\tau} d_{\tau}.
\end{equation*}
The next result gives explicit expressions for $\Delta_{\tau}$ in terms of powers of $\tau$. Here $\nabla$ is the Levi-Civita connection and $\nabla \rho$ is the metric gradient of $\rho$.

\begin{lemma} \label{lemma:delta_tau}
If $u$ is a $0$-form or $1$-form, then 
\begin{eqnarray*}
 & \Delta_{\tau} u = \tau^2 \langle d\rho, d\rho \rangle u - \tau [2 \nabla_{\nabla \rho} u + (\Delta \rho) u] + \Delta u, & \\
 & \Delta_{\tau} *u = *\left\{\tau^2 \langle d\rho, d\rho \rangle u - \tau [2 \nabla_{\nabla \rho} u +  (\Delta \rho) u] + \Delta u\right\}. & 
\end{eqnarray*}
\end{lemma}
\begin{proof}
The first identity for $0$-forms is a straightforward computation. If $u$ is a $1$-form, we have 
\begin{multline*}
-\Delta_{\tau} u = (d-\tau\,d\rho\wedge)(\delta u + \tau \langle d\rho, u \rangle) + (\delta - \tau * d\rho \wedge *)(du - \tau u \,d\rho) \\
 = -\Delta u + \tau [d\langle d\rho,u \rangle - (\delta u) d\rho - \delta(d\rho \wedge u) - *d\rho \wedge *du] - \tau^2 \langle d\rho,d\rho \rangle u. 
\end{multline*}
The identities in Section \ref{sec:notation} and a computation in normal coordinates show that 
\begin{equation*}
d\langle d\rho,u \rangle - (\delta u) d\rho - \delta(d\rho \wedge u) - *d\rho \wedge *du = 2 \nabla_{\nabla \rho} u + (\Delta \rho) u.
\end{equation*}
This proves the first identity for $1$-forms. The Hodge star commutes with $\Delta_{\tau}$ since it commutes with $\Delta$, and second identity follows.
\end{proof}

We write $A = \left( \begin{array}{cc|cc} A^0 & *B^1 & *B^0 & A^1 \end{array} \right)^t$ where $A^0, B^0$ are $0$-forms and $A^1, B^1$ are $1$-forms. Using Lemma \ref{lemma:delta_tau}, the WKB construction for solutions to \eqref{schrodinger_t} having the form \eqref{z_form} results in the following equations in $T$:
\begin{eqnarray}
 & \langle d\rho, d\rho \rangle = 0, & \label{wkb_eikonal} \\
 & 2 \nabla_{\nabla \rho} A^j + (\Delta \rho) A^j = 0 \quad (j=0,1), & \label{wkb_transport1} \\
 & 2 \nabla_{\nabla \rho} B^j + (\Delta \rho) B^j = 0 \quad (j=0,1), & \label{wkb_transport2} \\
 & e^{\tau \rho} (-\Delta - k^2 + Q) e^{-\tau \rho} R = (\Delta+k^2-Q)A. & \label{wkb_correction}
\end{eqnarray}

We follow the construction in \cite[Section 5]{DKSaU} and employ special coordinates to solve these equations. Considering the real and imaginary parts separately, the first equation \eqref{wkb_eikonal} reads 
\begin{equation*}
\abs{d\psi}^2 = \abs{d\varphi}^2, \quad \langle d\psi, d\varphi \rangle = 0.
\end{equation*}
Recall that $\varphi(x) = x_1$. Choose a point $p \in \tilde{M}_0 \smallsetminus M_0$, and let $(r,\theta)$ be polar normal coordinates in $(\tilde{M}_0,g_0)$ with center $p$. Then $r$ is smooth in $M_0$, and we obtain a solution $\psi$ by setting 
\begin{equation*}
\psi(x_1,r,\theta) = r.
\end{equation*}
Note that in the $(x_1,r,\theta)$ coordinates one has in $T$ 
\begin{equation*}
g(x_1,r,\theta) = \begin{pmatrix} 1 & & \\ & 1 & \\ & & m(r,\theta) \end{pmatrix}
\end{equation*}
where $m = \abs{g}$ is smooth. We write 
\begin{equation*}
\dbar = \frac{1}{2} \left(\frac{\partial}{\partial x_1} + i \frac{\partial}{\partial r} \right).
\end{equation*}
The following result gives solutions to the transport equations \eqref{wkb_transport1}--\eqref{wkb_transport2}.

\begin{lemma}
Assume the above notations.
\begin{enumerate}
\item 
If $a$ is a $0$-form, then $2 \nabla_{\nabla \rho} a + (\Delta \rho) a = 0$ iff $\dbar (\abs{g}^{1/4} a) = 0$.
\item 
If $\eta$ is a $1$-form, then $2 \nabla_{\nabla \rho} \eta + (\Delta \rho) \eta = 0$ iff $\eta = a_1 \,dx^1 + a_r \,dr + a_{\theta} \,d\theta$ with $\dbar (\abs{g}^{1/4} a_1) = \dbar (\abs{g}^{1/4} a_r) = \dbar(\abs{g}^{-1/4} a_{\theta}) = 0$.
\end{enumerate}
\end{lemma}
\begin{proof}
We have 
\begin{equation*}
\rho = x_1 + ir, \qquad \nabla \rho = 2 \dbar, \qquad \Delta \rho = \dbar \,\log \,\abs{g}.
\end{equation*}
The equation for $0$-forms is $4 \dbar a + \left( \dbar \,\log\,\abs{g} \right) a = 0$, which proves the first part. For the second part, the form of $g$ shows that 
\begin{equation*}
\Gamma_{1k}^l = 0, \qquad \Gamma_{rk}^l = \left\{ \begin{array}{cl} \frac{1}{2} \partial_r(\log\,\abs{g}), & k = l = \theta, \\ 0, & \text{otherwise}. \end{array} \right.
\end{equation*}
Consequently $\nabla_{\partial_1} \,dx^j = 0$ for all $j$ and $\nabla_{\partial_r} \,dx^1 = \nabla_{\partial_r} \,dr = 0$, $\nabla_{\partial_r} \,d\theta = -\frac{1}{2} \partial_r(\log\,\abs{g}) \,d\theta$. The result follows by noting that 
\begin{multline*}
\nabla_{\nabla \rho} (a_1 \,dx^1 + a_r \,dr + a_{\theta} \,d\theta) \\
 = (2\dbar a_1) \,dx^1 + (2\dbar a_r) \,dr + (2\dbar a_{\theta} - \frac{i}{2} \partial_r(\log\,\abs{g}) a_{\theta}) \,d\theta
\end{multline*}
where $\frac{i}{2} \partial_r(\log\,\abs{g}) = \dbar(\log\,\abs{g})$.
\end{proof}

We are now ready to give the construction of complex geometrical optics solutions to the Schr\"odinger equation.

\begin{prop} \label{prop:schrodinger_solution}
Let $(M_0,g_0) \subset \subset (\tilde{M}_0,g_0)$ be two simple $2$-manifolds, and consider the cylinders $T = \mR \times M_0$ and $\tilde{T} = \mR \times \tilde{M}_0$ equipped with the metric $g$ given by \eqref{metric_form}. Let $k \geq 0$ be a constant, let $\delta > 1/2$, and let $Q$ be a bounded linear operator $L^2_{-\delta}(\Omega T) \to L^2_{\delta}(\Omega T)$. There exists $\tau_0 \geq 1$ with the following property: if 
\begin{equation*}
\abs{\tau} \geq \tau_0 \quad \text{and} \quad \tau^2 + k^2 \notin \text{Spec}(-\Delta_{x'}^0) \cup \text{Spec}(-\Delta_{x'}^1),
\end{equation*}
and if $p, \lambda, a_0, a_1, b_0, b_1$ are any parameters such that 
\begin{eqnarray*}
 & \text{$p$ is a point in $\tilde{M}_0 \smallsetminus M_0$ and $\lambda > 0$ is a constant}, & \\
 & \text{$(r,\theta)$ are polar normal coordinates in $\tilde{M}_0$ with center $p$}, & \\ 
 & \text{$a_l, b_l$ are smooth functions in $T$ of the form $e^{i\lambda x_1} w(x')$}, & \\
 & \text{$(\partial_1 + i \partial_r) a_l = (\partial_1+i\partial_r) b_l = 0$ in $T$}, & 
\end{eqnarray*}
then the equation $(-\Delta-k^2+Q) Z = 0$ in $T$ has a unique solution 
\begin{equation} \label{schrodinger_solution_form}
Z = e^{-\tau(x_1+ir)} \left[ \abs{g}^{-1/4} \left( \begin{array}{c} a_0 \\ b_1 *dx^1 \\ \hline b_0 *1 \\ a_1 \,dx^1 \end{array} \right) + R \right]
\end{equation}
where $R \in H^2_{-\delta,\text{R}}(\Omega T)$. The remainder $R$ satisfies $\norm{R}_{L^2_{-\delta}(\Omega T)} \leq C \abs{\tau}^{-1}$ with $C$ independent of $\tau$.
\end{prop}
\begin{proof}
Take $\rho = x_1+ir$ and $A^0 = \abs{g}^{-1/4} a_0$, $A^1 = \abs{g}^{-1/4} a_1 \,dx^1$, and also $B^0 = \abs{g}^{-1/4} b_0$, $B^1 = \abs{g}^{-1/4} b_1 \,dx^1$. It follows from the discussion above that equations \eqref{wkb_eikonal}--\eqref{wkb_transport2} are satisfied, and that $Z$ solves $(-\Delta-k^2+Q) Z = 0$ iff 
\begin{equation} \label{wkb_correction_f}
e^{\tau \rho} (-\Delta - k^2 + Q) e^{-\tau \rho} R = F
\end{equation}
where $F = (\Delta+k^2-Q)A$.

The form of $a_l$ and $b_l$ implies that $A \in L^2_{-\delta}(\Omega T)$, so $QA \in L^2_{\delta}(\Omega T)$, and also that $(\Delta-k^2)A = e^{i\lambda x_1} \tilde{F}_2(x') \in L^2_{-\delta}(\Omega T)$ where $\tilde{F}_2 \in L^{\infty}(\Omega T)$, $\nabla_{\partial_1} \tilde{F}_2 = 0$. The latter fact follows from the formula for $\Delta$ acting on $\Omega T$ computed in the proof of Proposition \ref{prop:normestimate_uniqueness_potential_forms}. Thus we have a decomposition $F = F_1 + F_2$ as in Proposition \ref{prop:normestimate_uniqueness_potential_forms}, and that result shows that there is a unique solution $R \in H^2_{-\delta,\text{R}}(\Omega T)$ to \eqref{wkb_correction_f} with the required estimate if $\abs{\tau}$ is large and outside a discrete set.
\end{proof}

We now prove the main result on complex geometrical optics solutions. For part (a) we need to use the uniqueness of the solutions above to conclude that $Y^0 = Y^3 = 0$. Part (b) is  in fact much easier since it is enough to construct solutions to a Dirac equation in $M$ without worrying about the vanishing of scalar parts.

\begin{proof}[Proof of Theorem \ref{thm:cgo_main}(a)]
Assume the conditions in Theorem \ref{thm:cgo_main}, and extend $\eps$ and $\mu$ smoothly as constants into $\tilde{T}$. Then $Q$ satisfies the assumption in Proposition \ref{prop:schrodinger_solution}, and that result guarantees the existence of a solution $Z$ to $(-\Delta-k^2+Q)Z = 0$ in $T$ of the form \eqref{schrodinger_solution_form} with $R \in H^2_{-\delta,\text{R}}(T)$ and $\norm{R}_{L^2_{-\delta}(\Omega T)} \leq C \abs{\tau}^{-1}$. Setting $Y = (P+k-W^t)Z$, Lemma \ref{lemma:reduction_schrodinger} shows that $Y$ solves $(P-k+W)Y = 0$ in $T$.

The main point is to show that $Y^0 = Y^3 = 0$. For this we use an idea appearing in \cite{OS}. By Lemma \ref{lemma:reduction_schrodinger} we have $(-\Delta-k^2+Q')Y = 0$. Looking at the $0$-form and $3$-form parts and using the special form of $Q'$, the equation decouples and we obtain the following equations in $T$: 
\begin{align*}
 &(-\Delta-k^2+q^0) Y^0 = 0, \\
 &(-\Delta-k^2+q^3) *Y^3 = 0,
\end{align*}
where $q^0$ and $q^3$ are smooth potentials with compact support in $M^{\text{int}}$, 
\begin{align*}
q^0 &= k^2 - \kappa^2 - \frac{1}{2} \Delta \beta + \frac{1}{4} \langle d\beta,d\beta \rangle, \\ 
q^3 &= k^2 - \kappa^2 - \frac{1}{2} \Delta \alpha + \frac{1}{4} \langle d\alpha,d\alpha \rangle.
\end{align*}
Now, writing $\rho = x_1 + ir$ and $Z = e^{-\tau \rho}(A+R)$, $Y^0$ has the form 
\begin{align*}
Y^0 &= ((P+k-W^t)Z)^0 = e^{-\tau \rho}\left( \left(-\frac{1}{i} \delta_{\tau} + k - W^t \right)(A+R) \right)^0 \\
 &= e^{-\tau \rho} \left( y_0 + r_0 \right).
\end{align*}
Here we have written 
\begin{align*}
y_0 &= -\frac{1}{i} \delta_{\tau} A^1 + k A^0, \\
r_0 &= -(W^t(A+R))^0 - \frac{1}{i} \delta_{\tau} R^1 + k R^0.
\end{align*}
Since $R \in H^2_{-\delta,\text{R}}(\Omega T)$ we see that $r_0 \in H^1_{-\delta}(T)$ and $r_0|_{\partial T} = 0$.

We will choose $A^0, A^1$ so that 
\begin{equation} \label{y0_condition}
y_0 \equiv 0.
\end{equation}
Then $e^{-i\tau r} r_0$ will be a solution in $H^1_{-\delta,0}(T)$ of the equation 
\begin{equation*}
e^{\tau x_1}(-\Delta-k^2+q^0) e^{-\tau x_1} (e^{-i\tau r} r_0) = 0 \quad \text{in } T,
\end{equation*}
and the uniqueness part in Proposition \ref{prop:normestimate_uniqueness} will show that $r_0 = 0$, so also $Y^0 = 0$, if $\abs{\tau}$ is sufficiently large. To obtain \eqref{y0_condition} we make the choices $A^0 = \abs{g}^{-1/4} a_0$, $A^1 = \abs{g}^{-1/4} a_1 \,dx^1$ where 
\begin{align}
a_1 &= \frac{ik s_0}{\tau} e^{i\lambda(x_1+ir)} \chi(\theta), \label{a1_choice} \\
a_0 &= \frac{1}{ik} \delta_{\tau} (a_1 \,dx^1). \label{a0_choice}
\end{align}
This is consistent with Proposition \ref{prop:schrodinger_solution} since $(\partial_1 + i\partial_r) a_1 = 0$ and since 
\begin{equation} \label{a0_expression}
a_0 = \frac{1}{ik} (-\partial_1 a_1 + \tau a_1) = \frac{\tau-i\lambda}{ik} a_1
\end{equation}
so also $(\partial_1+i\partial_r)a_0 = 0$, and both $a_0$ and $a_1$ are of the form $e^{i\lambda x_1} w(x')$. Now \eqref{y0_condition} holds because 
\begin{equation*}
y_0 = -\frac{1}{i} \delta_{\tau} A^1 + k A^0 = -\frac{1}{i} \delta_{\tau} (\abs{g}^{-1/4} a_1 \,dx^1) + k \abs{g}^{-1/4} \frac{1}{ik} \delta_{\tau} (a_1 \,dx^1) = 0.
\end{equation*}

We have established that $Y^0 = 0$ given the choices \eqref{a1_choice}--\eqref{a0_choice}. A similar computation for $Y^3$, with 
\begin{align*}
b_1 &= \frac{ik t_0}{\tau} e^{i\lambda(x_1+ir)} \chi(\theta), \\
b_0 &= \frac{1}{ik} \delta_{\tau} (b_1 \,dx^1),
\end{align*}
shows that $Y^3 = 0$. Finally, we note that $Z$ is of the form 
\begin{equation*}
Z = e^{-\tau \rho} [ \abs{g}^{-1/4} \left( \begin{array}{cc|cc} a_0 & *b_1 \,dx^1 & *b_0 & a_1 \,dx^1 \end{array} \right)^t + R ]
\end{equation*}
where $a_0 = s_0 e^{i\lambda(x_1+ir)} \chi(\theta) + O(\abs{\tau}^{-1})$ by \eqref{a0_expression}, and $a_1 = O(\abs{\tau}^{-1})$. We have written $O(\abs{\tau}^{-1})$ for quantities whose $L^2(M)$ norm is $\leq C \abs{\tau}^{-1}$. Similar expressions are true for $b_0$ and $b_1$. This shows that $Z$ has the required form \eqref{zform_parta} with $\norm{R}_{L^2(\Omega M)} \leq C \abs{\tau}^{-1}$.

(Note that $a_j$ and $b_j$ are mildly $\tau$-dependent, but their $W^{l,\infty}(T)$ norms are bounded uniformly in $\tau$ which implies that final constant $C$ does not depend on $\tau$).
\end{proof}

\begin{proof}[Proof of Theorem \ref{thm:cgo_main}(b)]
Again, assume the conditions in Theorem \ref{thm:cgo_main} and extend $\eps$ and $\mu$ smoothly as constants into $\tilde{T}$. Let $\hat{Q}$ be the potential in Lemma \ref{lemma:reduction_schrodinger}. We look for a solution to $(-\Delta-k^2+\hat{Q}) Z = 0$ in $M$ of the form 
\begin{equation*}
Z = e^{-\tau(x_1+ir)} \left[ \abs{g}^{-1/4} \left( \begin{array}{c} 0 \\ t_0 * dx^1 \\ \hline 0 \\ s_0 \,dx^1 \end{array} \right) + R \right].
\end{equation*}
Write $\rho = x_1+ir$ and $A = \abs{g}^{-1/4} \left( \begin{array}{cc|cc} 0 & t_0 * dx^1 & 0 & s_0 \,dx^1 \end{array} \right)^t$. Following the WKB construction, it is enough to solve 
\begin{equation*}
e^{\tau \rho}(-\Delta-k^2+\hat{Q})(e^{-\tau \rho} R) = (\Delta+k^2-\hat{Q})A \quad \text{in $M$}.
\end{equation*}
Define $F \in L^2_{\delta}(\Omega T)$ with $F = e^{-i\tau r}(\Delta+k^2-\hat{Q})A$ in $M$ and $F = 0$ in $T \smallsetminus M$, and let $e^{-i\tau r} R$ be a solution provided by Proposition \ref{prop:normestimate_uniqueness_potential_forms} of the equation $e^{\tau x_1}(-\Delta-k^2+\hat{Q})(e^{-\tau x_1} [e^{-i\tau r} R]) = F$ in $T$. This gives the required solution $Z$ in $M$ satisfying $\norm{R}_{L^2(\Omega M)} \leq C \abs{\tau}^{-1}$ and $\norm{R}_{H^1(\Omega M)} \leq C$.

We set 
\begin{equation*}
Y = \frac{1}{\tau}(P+k-\bar{W})Z.
\end{equation*}
By Lemma \ref{lemma:reduction_schrodinger} this satisfies $(P-k+W^*) Y = 0$ in $M$, and 
\begin{align*}
Y &= e^{-\tau \rho} \left( \frac{1}{i} d_{\tau} - \frac{1}{i} \delta_{\tau} + k - \bar{W} \right) (\tau^{-1} A + \tau^{-1} R) \\
 &= e^{-\tau \rho} \bigg[ - \frac{\abs{g}^{-1/4}}{i} d\rho \wedge (s_0 \,dx^1 + t_0 *dx^1) \\
 &\qquad \qquad - \frac{\abs{g}^{-1/4}}{i} *d\rho \wedge * (s_0 \,dx^1 - t_0 *dx^1) + O(\abs{\tau}^{-1}) \bigg] \\
 &= e^{-\tau \rho} \left[ \abs{g}^{-1/4}  \left( \begin{array}{c} i s_0 \\ s_0 \,dx^1 \wedge dr \\ \hline i t_0 *1 \\ -t_0 *dx^1 \wedge dr \end{array} \right) + O(\abs{\tau}^{-1}) \right].
\end{align*}
Here $O(\abs{\tau}^{-1})$ denotes a quantity whose $L^2(\Omega M)$ norm is $\leq C \abs{\tau}^{-1}$. The result follows upon replacing $s_0$ by $-is_0$ and $t_0$ by $-i t_0$.
\end{proof}

\section{Recovering the coefficients} \label{sec:recovery}

We shall use the complex geometrical optics solutions constructed in Theorem \ref{thm:cgo_main} to prove Theorem \ref{thm:main}. The first step is a reduction to the case where the conformal factor in the metric is equal to one. We write $\Lambda = \Lambda_{g,\eps,\mu}$ for the admittance map in $(M,g)$ with coefficients $\eps$ and $\mu$.

\begin{lemma} \label{lemma:maxwell_conformal}
Let $(M,g)$ be a compact Riemannian $3$-manifold with smooth boundary, and let $\eps$ and $\mu$ satisfy \eqref{paramcond1}--\eqref{paramcond3p}. If $c$ is any smooth positive function on $M$, then 
\begin{equation*}
\Lambda_{cg,\eps,\mu} = \Lambda_{g,c^{1/2} \eps,c^{1/2}\mu}.
\end{equation*}
\end{lemma}
\begin{proof}
Follows by noting that $*_{cg} u = c^{3/2-k} *_g u$ for a $k$-form $u$, so that a pair $(E,H)$ satisfies \eqref{maxwell_equations} with metric $cg$ and coefficients $\eps$ and $\mu$ iff it satisfies \eqref{maxwell_equations} with metric $g$ and coefficients $c^{1/2} \eps$ and $c^{1/2} \mu$.
\end{proof}

\begin{proof}[Proof of Theorem \ref{thm:main}]
According to Lemma \ref{lemma:integral_identity} we may assume that \eqref{paramcond5} holds and the identity \eqref{int_id} is valid. By the definition of admissible manifolds, there are global coordinates $x = (x_1,x')$ such that $g$ has the form 
\begin{equation*}
g(x) = c(x) \begin{pmatrix} 1 & 0 \\ 0 & g_0(x') \end{pmatrix}.
\end{equation*}
If $\Lambda_{g,\eps_1,\mu_1} = \Lambda_{g,\eps_2,\mu_2}$, then also $\Lambda_{c^{-1}g,c^{1/2}\eps_1,c^{1/2}\mu_1} = \Lambda_{c^{-1} g,c^{1/2}\eps_2,c^{1/2}\mu_2}$ by Lemma \ref{lemma:maxwell_conformal}. This shows that we may also assume $c \equiv 1$.

By Theorem \ref{thm:cgo_main}, if $\tau$ is outside a discrete set and $\abs{\tau}$ is sufficiently large, and if $\chi(\theta)$ is a smooth function and $\lambda > 0$ and $s_0$ and $t_0$ are real numbers, then there exist $Z_1$ and $Y_2$ in $H^2(\Omega M)$ satisfying the conditions in Lemma \ref{lemma:integral_identity} and having the form 
\begin{align*}
Z_1 &= e^{-\tau(x_1+ir)} \left[ \abs{g}^{-1/4} e^{i\lambda(x_1+ir)} \chi(\theta) \left( \begin{array}{c} s_0 \\ 0 \\ \hline t_0 *1 \\ 0 \end{array} \right) + R_1 \right], \\
Y_2 &= e^{\tau(x_1-ir)} \left[ \abs{g}^{-1/4} \left( \begin{array}{c} s_0 \\ is_0 \,dx^1 \wedge dr \\ \hline t_0 *1 \\ -it_0 *dx^1 \wedge dr \end{array} \right) + R_2 \right]
\end{align*}
where $\norm{R_j}_{L^2(\Omega M)} \leq C \abs{\tau}^{-1}$ with $C$ independent of $\tau$. In the second solution, we used $-r$ instead of $r$ as the solution of the eikonal equation.

By Lemma \ref{lemma:integral_identity}, these solutions satisfy the identity 
\begin{equation*}
\int_M \langle (Q_1-Q_2)Z_1, \bar{Y}_2 \rangle \,dV = 0.
\end{equation*}
Letting $\tau \to \infty$ outside the discrete set and using the estimates for $R_j$, we obtain in terms of the $x = (x_1,r,\theta)$ coordinates that 
\begin{equation*}
\int_M \left\langle (Q_1-Q_2) \left( \begin{array}{c} s_0 \\ 0 \\ \hline t_0 *1 \\ 0 \end{array} \right), \left( \begin{array}{c} s_0 \\ -is_0 \,dx^1 \wedge dr \\ \hline t_0 *1 \\ it_0 *dx^1 \wedge dr \end{array} \right) \right\rangle e^{i\lambda(x_1+ir)} \chi(\theta) \,dx = 0.
\end{equation*}
Let $q_{\alpha}$ and $q_{\beta}$ be the elements of $Q_1-Q_2$, interpreted as a $8 \times 8$ matrix, which correspond to the $(1,1)$th and $(5,5)$th elements, respectively. By Lemma \ref{lemma:reduction_schrodinger}
\begin{align*}
q_{\alpha} &= \frac{1}{2}\Delta(\alpha_1-\alpha_2) + \frac{1}{4} \langle d\alpha_1,d\alpha_1 \rangle - \frac{1}{4} \langle d\alpha_2,d\alpha_2 \rangle - \omega^2(\eps_1 \mu_1 - \eps_2 \mu_2), \\
q_{\beta} &= \frac{1}{2}\Delta(\beta_1-\beta_2) + \frac{1}{4} \langle d\beta_1,d\beta_1 \rangle - \frac{1}{4} \langle d\beta_2,d\beta_2 \rangle - \omega^2(\eps_1 \mu_1 - \eps_2 \mu_2).
\end{align*}
With the two choices $(s_0,t_0) = (1,0)$ and $(s_0,t_0) = (0,1)$, the special form of $Q_1$ and $Q_2$ in Lemma \ref{lemma:reduction_schrodinger} shows that we obtain the two identities 
\begin{eqnarray*}
 & \int_M e^{i\lambda(x_1+ir)} \chi(\theta) q_{\alpha}(x) \,dx = 0, & \\
 & \int_M e^{i\lambda(x_1+ir)} \chi(\theta) q_{\beta}(x) \,dx = 0. & 
\end{eqnarray*}

We extend $q_{\alpha}$ and $q_{\beta}$ to be zero in $T \smallsetminus M$, where $T \supset \supset M$ is as in the definition of admissible manifolds. Then the integrals above may be taken over $T = \mR \times M_0$. Varying $\chi(\theta)$, it follows that for all $\theta$ we have 
\begin{equation*}
\int_0^{\infty} e^{-\lambda r} \left[ \int_{-\infty}^{\infty} e^{i\lambda x_1} q_{\alpha}(x_1,r,\theta) \,dx_1 \right] \,dr = 0
\end{equation*}
and similarly for $q_{\beta}$. Now, since $(r,\theta)$ are polar normal coordinates in $M_0$, the curves $r \mapsto (r,\theta)$ are geodesics in $M_0$. Denoting the expression in brackets by $f_{\alpha}(r,\theta)$ and varying the point $p$ in Theorem \ref{thm:cgo_main} and varying $\theta$, we obtain that 
\begin{equation*}
\int_0^{\infty} f_{\alpha}(\gamma(r)) \exp\left[ -\int_0^r \lambda \,ds \right] \,dr = 0
\end{equation*}
for all geodesics $\gamma$ in $M_0$ which begin and end at points of $\partial M_0$. This shows the vanishing of the geodesic ray transform of the function $f_{\alpha}$ with constant attenuation $-\lambda$. For more details we refer to \cite[Section 7]{DKSaU}. In particular, the injectivity result given by Theorem 7.1 in \cite{DKSaU} implies that $f_{\alpha} \equiv 0$ for all positive $\lambda$ which are sufficiently small. Thus 
\begin{equation*}
\int_{-\infty}^{\infty} e^{i\lambda x_1} q_{\alpha}(x_1,r,\theta) \,dx_1 = 0
\end{equation*}
for such $\lambda$ and for all $r$ and $\theta$. Since $q_{\alpha}$ is compactly supported in $x_1$, the Paley-Wiener theorem shows that $q_{\alpha} \equiv 0$ in $M$. We obtain $q_{\beta} \equiv 0$ in $M$ by the exact same argument.

We have arrived at the following two equations in $M$:
\begin{eqnarray*}
 & -\frac{1}{2} \Delta(\alpha_1-\alpha_2) - \frac{1}{4}\langle d(\alpha_1+\alpha_2),d(\alpha_1-\alpha_2) \rangle + \omega^2(\eps_1 \mu_1 - \eps_2 \mu_2) = 0, & \\
 & -\frac{1}{2} \Delta(\beta_1-\beta_2) - \frac{1}{4}\langle d(\beta_1+\beta_2),d(\beta_1-\beta_2) \rangle + \omega^2(\eps_1 \mu_1 - \eps_2 \mu_2) = 0. & 
\end{eqnarray*}
Let $u = (\eps_1/\eps_2)^{1/2}$ and $v = (\mu_1/\mu_2)^{1/2}$. Then $\frac{1}{2}(\alpha_1-\alpha_2) = \log\,u$, and the equations become 
\begin{eqnarray*}
 & -\Delta(\log\,u) - (\eps_1 \eps_2)^{-1/2} \langle d(\eps_1 \eps_2)^{1/2}, d(\log\,u) \rangle + \omega^2(\eps_1 \mu_1 - \eps_2 \mu_2) = 0, & \\
 & -\Delta(\log\,v) - (\mu_1 \mu_2)^{-1/2} \langle d(\mu_1 \mu_2)^{1/2}, d(\log\,v) \rangle + \omega^2(\eps_1 \mu_1 - \eps_2 \mu_2) = 0. & 
\end{eqnarray*}
Multiplying the first equation by $(\eps_1 \eps_2)^{1/2}$ and the second by $(\mu_1 \mu_2)^{1/2}$, and using that $\delta(a\nabla w) = -a \Delta w - \langle da, dw \rangle$, we see that $u$ and $v$ satisfy the semilinear elliptic system 
\begin{eqnarray*}
 & \delta(\eps_2 d u) + \omega^2 \eps_2^2 \mu_2 (u^2 v^2 - 1)u = 0, & \\
 & \delta(\mu_2 d v) + \omega^2 \eps_2 \mu_2^2 (u^2 v^2 - 1)v = 0. &  
\end{eqnarray*}
The condition \eqref{paramcond5} ensures that one has $u = 1$ and $v = 1$ near $\partial M$ . Also, the above equations imply that the pair $(\tilde{u},\tilde{v}) = (1,1)$ is a solution of the semilinear system in all of $M$. By Theorem \ref{thm:app_uniquecontinuation} unique continuation holds for this system, and we obtain $u \equiv 1$ and $v \equiv 1$ in $M$. This proves that $\eps_1 \equiv \eps_2$ and $\mu_1 \equiv \mu_2$ in $M$ as required.
\end{proof}

We now prove Theorem \ref{thm:main2}. The treatment below follows \cite{KLS}. Let $\Omega \subseteq \mR^3$ be a bounded open set with smooth boundary, and let $\eps$ and $\mu$ be symmetric positive definite $(1,1)$-tensors on $\closure{\Omega}$. We equip $\closure{\Omega}$ with the Euclidean metric $e$. The Maxwell equations \eqref{maxwell_equations_omega} can be written as 
\begin{equation} \label{maxwell_omega_curl}
\left\{ \begin{array}{rl}
\text{curl}_{e}(\vec{E}) &\!\!\!= i\omega \mu \vec{H}, \\
\text{curl}_{e}(\vec{H}) &\!\!\!= -i\omega \eps \vec{E}.
\end{array} \right.
\end{equation}
Here 
\begin{equation*}
\text{curl}_e(\vec{X}) = (*_e d \vec{X}^{\flat})^{\sharp}
\end{equation*}
with the flat and sharp operators taken with respect to $e$.

For the vector fields $\vec{E} = (E_1,E_2,E_3)$ and $\vec{H} = (H_1,H_2,H_3)$, let $E = \vec{E}^{\flat} = E_j \,dx^j$ and $H = \vec{H}^{\flat} = H_j \,dx^j$ be the corresponding $1$-forms. To write \eqref{maxwell_omega_curl} in a form similar to \eqref{maxwell_equations}, it is enough to find Riemannian metrics $g_{\eps}$ and $g_{\mu}$ so that 
\begin{equation*}
*_e (\eps \vec{E})^{\flat} = *_{g_{\eps}} E, \quad *_e (\mu \vec{H})^{\flat} = *_{g_{\mu}} H.
\end{equation*}
These conditions will be satisfied if we choose in local coordinates 
\begin{equation} \label{traveltimemetric_def}
g_{\eps}^{jk} = \frac{1}{\det(\eps)} \eps^k_j, \quad g_{\mu}^{jk} = \frac{1}{\det(\mu)} \mu^k_j.
\end{equation}
Then \eqref{maxwell_omega_curl} is equivalent with 
\begin{equation} \label{maxwell_omega_gepsmu}
\left\{ \begin{array}{rl}
*_{g_{\mu}} dE &\!\!\!= i\omega H, \\
*_{g_{\eps}} dH &\!\!\!= -i\omega E.
\end{array} \right.
\end{equation}

We now use the assumption that $\eps$ and $\mu$ are in the same conformal class, so that $\mu = \alpha^2 \eps$ for some smooth positive function $\alpha$ on $\closure{\Omega}$. This allows to define a metric $g$ on $\closure{\Omega}$ by 
\begin{equation*}
g = \alpha^2 g_{\eps} = \alpha^{-2} g_{\mu}.
\end{equation*}
Since $*_{cg} u = c^{-1/2} *_g u$ for a $2$-form $u$, \eqref{maxwell_omega_gepsmu} is equivalent with 
\begin{equation} \label{maxwell_omega_4}
\left\{ \begin{array}{rl}
*_g dE &\!\!\!= i\omega \alpha H, \\
*_g dH &\!\!\!= -i\omega \alpha^{-1} E.
\end{array} \right.
\end{equation}
This is of the form \eqref{maxwell_equations}, and further the tangential boundary condition $tE = f$ is of the same form as \eqref{boundary_omega}. Since \eqref{maxwell_omega_4} is equivalent with \eqref{maxwell_equations_omega}, Theorem \ref{thm:app_wellposedness} implies that for $\omega$ outside a discrete set of frequencies the system \eqref{maxwell_equations_omega}--\eqref{boundary_omega} is uniquely solvable for a given boundary value $\vec{f}$. The admittance map $\Lambda$ for \eqref{maxwell_equations_omega} reduces to the map $\Lambda_{g,\alpha^{-1},\alpha}$.

Given this reduction, it is easy to prove the second main result of the paper.

\begin{proof}[Proof of Theorem \ref{thm:main2}]
Upon interpreting $\eps$ and $\mu$ as $(0,2)$-tensors by raising one index with respect to the Euclidean metric, the condition \eqref{traveltimemetric_def} implies that 
\begin{equation*}
\eps = \det(\eps) g_{\eps}^{-1}, \quad \mu = \det(\mu) g_{\mu}^{-1}.
\end{equation*}
From the assumption in the theorem, we know that there is an admissible metric $g$ and smooth positive functions $c_j, \tilde{c_j}$ on $\closure{\Omega}$ for which 
\begin{equation*}
g_{\eps_j} = c_j^2 g, \quad g_{\mu_j} = \tilde{c}_j^2 g.
\end{equation*}
Using these formulas in \eqref{maxwell_omega_gepsmu}, the Maxwell equations in $\Omega$ with coefficients $\eps_j$ and $\mu_j$ are equivalent with 
\begin{equation*}
\left\{ \begin{array}{rl}
*_g dE &\!\!\!= i\omega \tilde{c}_j H, \\
*_g dH &\!\!\!= -i\omega c_j E.
\end{array} \right.
\end{equation*}
Since the admittance maps for the Maxwell equations in $\Omega$ coincide, it follows that $\Lambda_{g,c_1,\tilde{c}_1} = \Lambda_{g,c_2,\tilde{c}_2}$. From Theorem \ref{thm:main} we obtain $c_1 = c_2$ and $\tilde{c}_1 = \tilde{c}_2$, which implies that $g_{\eps_1} = g_{\eps_2}$ and $g_{\mu_1} = g_{\mu_2}$. By \eqref{traveltimemetric_def} 
\begin{equation*}
\frac{1}{\det(\eps_1)} \eps_1 = \frac{1}{\det(\eps_2)} \eps_2, \quad \frac{1}{\det(\mu_1)} \mu_1 = \frac{1}{\det(\mu_2)} \mu_2.
\end{equation*}
Taking determinants gives that $\det(\eps_1) = \det(\eps_2)$ and $\det(\mu_1) = \det(\mu_2)$. Consequently $\eps_1 \equiv \eps_2$ and $\mu_1 \equiv \mu_2$.
\end{proof}

\begin{remark}
Note that in the setting of \eqref{maxwell_equations}, if $\eps$ and $\mu$ are real valued, then a conformal scaling of the metric would reduce \eqref{maxwell_equations} to a system of the form \eqref{maxwell_omega_4}. Therefore, in Sections \ref{sec:reductions} and \ref{sec:solutions} it would be enough to consider the case where $\mu = \eps^{-1}$, which would simplify some of the arguments slightly.
\end{remark}

\appendix

\section{Wellposedness theory} \label{sec:wellposedness}

Let $(M,g)$ be a compact oriented Riemannian $3$-manifold with smooth boundary $\partial M$. Consider the Maxwell equations 
\begin{equation} \label{app_maxwell_equations}
\left\{ \begin{array}{rll}
*dE &\!\!\!= i\omega \mu H & \quad \text{in } M, \\
*dH &\!\!\!= -i\omega \eps E & \quad \text{in } M,
\end{array} \right.
\end{equation}
with the tangential boundary condition 
\begin{equation} \label{app_tang_condition}
tE = f \quad \text{on } \partial M.
\end{equation}

Here we assume that $\eps$ and $\mu$ are complex functions in $C^k(M)$ whose real parts are positive in $M$, and $\omega$ is a complex number. To describe the boundary condition in more detail, we introduce the $\text{Div}$-spaces
\begin{gather*}
H^s_{\text{Div}}(M) = \{ u \in H^s \Omega^1(M) \,;\, \text{Div}(tu) \in H^{s-1/2}(\partial M) \}, \\
T H^s_{\text{Div}}(\partial M) = \{ f \in H^s \Omega^1(\partial M) \,;\, \text{Div}(f) \in H^s(\partial M) \}.
\end{gather*}
There are Hilbert spaces with norms 
\begin{gather*}
\norm{u}_{H^s_{\text{Div}}(M)} = \norm{u}_{H^s \Omega^1(M)} + \norm{\text{Div}(tu)}_{H^{s-1/2}(\partial M)}, \\
\norm{f}_{T H^s_{\text{Div}}(\partial M)} = \norm{f}_{H^s \Omega^1(\partial M)} + \norm{\text{Div}(f)}_{H^s(\partial M)}.
\end{gather*}
It is easy to see that $t(H^s_{\text{Div}}(M)) = TH^{s-1/2}_{\text{Div}}(\partial M)$ for $s > 1/2$.

\begin{thm} \label{thm:app_wellposedness}
Let $\eps, \mu \in C^k(M)$, $k \geq 2$, be functions with positive real parts. There is a discrete subset $\Sigma$ of $\mC$ such that if $\omega$ is outside this set, then one has a unique solution $(E, H) \in H^k_{\text{Div}}(M) \times H^k_{\text{Div}}(M)$ of \eqref{app_maxwell_equations}--\eqref{app_tang_condition} given any $f \in TH^{k-1/2}_{\text{Div}}(\partial M)$. The solution satisfies 
\begin{equation*}
\norm{E}_{H^k_{\text{Div}}(M)} + \norm{H}_{H^k_{\text{Div}}(M)} \leq C \norm{f}_{TH^{k-1/2}_{\text{Div}}(\partial M)}
\end{equation*}
with $C$ independent of $f$. In particular, if $\eps, \mu \in C^{\infty}(M)$ and $f \in \Omega^1(\partial M)$, then there is a unique solution $(E,H) \in \Omega^1(M) \times \Omega^1(M)$.
\end{thm}

The existence of a solution will be proved by the well-known variational method as in \cite{Costabel}, \cite{Leis}. We proceed to describe this method. The first step is to solve for $H$ in the first line of \eqref{app_maxwell_equations} and to substitute this on the second line, which leads to the second order equation 
\begin{equation*}
\delta(\mu^{-1} dE) - \omega^2 \eps E = 0.
\end{equation*}
However, this equation does not imply the divergence condition $\delta(\eps E) = 0$ which is necessary for solutions of \eqref{app_maxwell_equations}. To make up for this, we consider the modified equation 
\begin{equation*}
\delta(\mu^{-1} dE) + s \eps^{-1} d \delta (\eps E) - \omega^2 \eps E = 0
\end{equation*}
where $s$ is a positive real number. The condition $\delta(\eps E) = 0$ will follow later from the equation 
\begin{equation*}
s \delta(\eps^{-1} d [\delta(\eps E)]) - \omega^2 \delta(\eps E) = 0,
\end{equation*}
which is obtained by applying $\delta$ to the earlier equation.

To connect the present situation to boundary value problems for the Hodge Laplacian, we write $e = \eps E$ and note that $e$ should satisfy 
\begin{equation} \label{app_secondorder_e}
\delta(\mu^{-1} d(\eps^{-1} e)) + s \eps^{-1} d \delta e - \omega^2 e = 0.
\end{equation}
Taking the $L^2$ inner product of this with $\bar{\eps} \tilde{e}$ for a $1$-form $\tilde{e}$, and assuming the relative boundary conditions (see \cite[Section 5.9]{T1})
\begin{equation*}
te = t\tilde{e} = 0 \text{ and } t(\delta e) = t(\delta \tilde{e}) = 0
\end{equation*}
or the absolute boundary conditions
\begin{equation*}
t(*e) = t(*\tilde{e}) = 0 \text{ and } t(\delta *e) = t(\delta *\tilde{e}) = 0,
\end{equation*}
we end up with the following bilinear form for solving the Maxwell equations.

\begin{definition}
If $e, \tilde{e} \in H^1_b \Omega^1 (M)$ (where $b=R$ or $b=A$), we define 
\begin{equation*}
B(e,\tilde{e}) = (\mu^{-1} d(\eps^{-1} e) | d(\bar{\eps} \tilde{e})) + s (\delta e | \delta \tilde{e}).
\end{equation*}
Here we have used the spaces 
\begin{align*}
H^1_R \Omega^1(M) &= \{ u \in H^1 \Omega^1(M) \,;\, tu = 0 \}, \\
H^1_A \Omega^1(M) &= \{ u \in H^1 \Omega^1(M) \,;\, t(*u) = 0 \}.
\end{align*}
If $k \geq 2$ we define $H^k_R \Omega^1 (M) = \{ u \in H^k \Omega^1(M) \,;\, tu = t(\delta u) = 0 \}$ and $H^k_A \Omega^1 (M) = \{ u \in H^k \Omega^1(M) \,;\, t(*u) = t(\delta *u) = 0 \}$.
\end{definition}

One defines weak solutions of \eqref{app_secondorder_e} in the usual way. The main point is the following solvability result.

\begin{prop} \label{app:wellposedness_secondorder}
Let $\eps$ and $\mu$ be functions in $C^1(M)$ with positive real parts, and let $s$ be a positive real number. There is a discrete set $\Sigma_s$ in $\mC$ such that if $\omega$ is outside this set, then for any $f \in (H^1_b \Omega^1(M))'$ the equation 
\begin{equation} \label{app_secondorder_e_inhomog}
\delta(\mu^{-1} d(\eps^{-1} e)) + s \eps^{-1} d \delta e - \omega^2 e = f
\end{equation}
has a unique solution $e \in H^1_b \Omega^1(M)$ ($b=R$ or $A$). One has 
\begin{equation*}
\norm{e}_{H^1} \leq C \norm{f}_{(H^1_b)'}.
\end{equation*}
\end{prop}
\begin{proof}
Clearly $B$ is a sesquilinear form on $H^1_b \Omega^1(M)$ with 
\begin{equation*}
\abs{B(e,\tilde{e})} \leq C \norm{e}_{H^1} \norm{\tilde{e}}_{H^1}.
\end{equation*}
We may write $B(e,e) = B_0(e,e) + B_1(e,e)$ where 
\begin{equation*}
B_0(e,e) = (\mu^{-1} de|de) + s (\delta e|\delta e),
\end{equation*}
and $B_1$ is a sesquilinear form such that $\abs{B_1(e,e)} \leq C \norm{e}_{L^2} \norm{e}_{H^1}$. It follows that  
\begin{equation*}
\re \,B(e,e) \geq c \norm{de}_{L^2}^2 + s \norm{\delta e}_{L^2}^2 - C \norm{e}_{L^2} \norm{e}_{H^1}.
\end{equation*}
We now invoke a Poincar\'e inequality for $1$-forms with relative or absolute boundary values: by \cite[Section 5.9]{T1} one has 
\begin{equation*}
\norm{u}_{H^1}^2 \leq C(\norm{u}_{L^2}^2 + \norm{du}_{L^2}^2 + \norm{\delta u}_{L^2}^2), \quad u \in H^1_b \Omega^1(M).
\end{equation*}
It follows that 
\begin{equation*}
\re \,B(e,e) \geq c \norm{e}_{H^1}^2 - C \norm{e}_{L^2}^2.
\end{equation*}

We have proved that for some $C_0 > 0$, the sesquilinear form $B + d$ is bounded and coercive on $H^1_b \Omega^1(M)$ for any $d \in L^{\infty}(M)$ with $\re(d) \geq C_0$. By the Lax-Milgram theorem there is a bounded linear operator $T: (H^1_b \Omega^1)' \to H^1_b \Omega^1$ which maps $f$ to the unique solution $e$ of \eqref{app_secondorder_e_inhomog} where $-\omega^2$ is replaced by the constant $C_1 = C_0/\min_{x \in M} \re(\eps)$. Now, $e$ solves \eqref{app_secondorder_e_inhomog} iff 
\begin{equation*}
(I - (C_1 + \omega^2)T)e = Tf.
\end{equation*}
The last equation has a unique solution iff either $(C_1+\omega^2)^{-1} \notin \text{Spec}(T)$ or $\omega^2 = -C_1$. The operator $T: H^1_b \Omega^1 \to H^1_b \Omega^1$ is compact by the compact embedding $H^1_b \Omega^1 \to L^2 \Omega^1$, and $0 \notin \text{Spec}(T)$, so $\text{Spec}(T)$ is discrete. Then the set 
\begin{equation*}
\Sigma_s = \{ \omega \in \mC \smallsetminus \{\pm i \sqrt{C_1} \} \,;\, (C_1+\omega^2)^{-1} \in \text{Spec}(T) \}
\end{equation*}
is also discrete and \eqref{app_secondorder_e_inhomog} is uniquely solvable for any $\omega \notin \mC \smallsetminus \Sigma_s$.
\end{proof}

Given the last result, higher order regularity for solutions follows in a similar way as for the Hodge Laplacian (for more details see \cite[Proposition 9.7]{T1} and the results mentioned there). 

\begin{prop} \label{app:wellposedness_secondorder_2}
Let $\eps$ and $\mu$ be functions in $C^k(M)$, $k \geq 2$, with positive real parts, and let $s > 0$. If $\omega \notin \Sigma_s$, then for any $f \in H^{k-2} \Omega^1(M)$ the equation \eqref{app_secondorder_e_inhomog} has a unique solution $e \in H^k_b \Omega^1(M)$ ($b=R$ or $A$), and 
\begin{equation*}
\norm{e}_{H^k} \leq C \norm{f}_{H^{k-2}}.
\end{equation*}
Further, if $\omega$ is any complex number and if $e \in H^1_b \Omega^1(M)$ solves \eqref{app_secondorder_e_inhomog} for some $f \in H^{k-2} \Omega^1(M)$, then $e \in H^k_b \Omega^1(M)$.
\end{prop}

Finally, we connect the above discussion to the Maxwell system and prove the main result.

\begin{proof}[Proof of Theorem \ref{thm:app_wellposedness}]
We take $\Sigma$ to be the set $\Sigma_1$ in Proposition \ref{app:wellposedness_secondorder}, and assume that $\omega \notin \Sigma$. As a technical preparation, we choose $s > 0$ so that $\omega^2/s$ is not an eigenvalue of the operator $u \mapsto \delta(\eps^{-1} du)$ defined on $H^1_0(M)$. We then have $\omega \notin \Sigma_s$, which may be seen as follows: if $e \in H^1_R \Omega^1(M)$ satisfies \eqref{app_secondorder_e}, then $e \in H^2_R \Omega^1(M)$ by Proposition \ref{app:wellposedness_secondorder_2}, and applying $\delta$ to both sides of \eqref{app_secondorder_e} shows that $u=\delta e$ is a solution in $H^1_0(M)$ of 
\begin{equation*}
s \delta (\eps^{-1} du) - \omega^2 u = 0.
\end{equation*}
By the choice of $s$ we have $u=0$, which implies that $e$ satisfies \eqref{app_secondorder_e} with $s=1$. Then $e = 0$ by the assumption $\omega \notin \Sigma_1$, and therefore $\omega \notin \Sigma_s$.

For uniqueness, if $(E,H) \in H^1_{\text{Div}}(M) \times H^1_{\text{Div}}(M)$ solve \eqref{app_maxwell_equations}--\eqref{app_tang_condition} with $f = 0$, then one has 
\begin{align*}
\delta(\mu^{-1} dE) &= \omega^2 \eps E, \\
\delta(\eps E) &= 0.
\end{align*}
It follows that $e = \eps E$ is in $H^1_R \Omega^1(M)$ and 
\begin{equation*}
\delta(\mu^{-1} d(\eps^{-1} e)) + \eps^{-1} d \delta e - \omega^2 e = 0.
\end{equation*}
Proposition \ref{app:wellposedness_secondorder} shows that $e = 0$, which implies $E = H = 0$.

Let us proceed to prove existence of solutions. Given $f \in T H^{k-1/2}_{\text{Div}}(\partial M)$, choose $E^0 \in H^k_{\text{Div}}(M)$ with $t E^0 = f$ and $t \delta (\eps E^0) = 0$. A computation in boundary normal coordinates shows that the extension map $f \mapsto E^0$ can be taken to be bounded and linear $TH^{k-1/2}_{\text{Div}}(\partial M) \to H^k_{\text{Div}}(M)$.

We let $e \in H^k_R \Omega^1(M)$ be the solution, given by Proposition \ref{app:wellposedness_secondorder_2}, of 
\begin{equation*}
\delta(\mu^{-1} d(\eps^{-1} e)) + s \eps^{-1} d\delta e - \omega^2 e = F
\end{equation*}
with $F = -\delta(\mu^{-1} d E^0) - s \eps^{-1} d\delta(\eps E^0) + \omega^2 \eps E^0 \in H^{k-2} \Omega^1(M)$. Now define $E = \eps^{-1} e + E^0$ and $H = \frac{1}{i\omega \mu} *dE$. With these conventions, $E$ satisfies the equation 
\begin{equation} \label{app:sol_e_eq}
\delta(\mu^{-1} dE) + s \eps^{-1} d\delta (\eps E) - \omega^2 \eps E = 0.
\end{equation}

We now claim that 
\begin{equation} \label{app:e_claim}
\delta (\eps E) = 0.
\end{equation}
In fact, by taking $\delta$ of both sides of \eqref{app:sol_e_eq}, the function $u = \delta(\eps E) \in H^1_0(M)$ satisfies $s \delta (\eps^{-1} d u) - \omega^2 u = 0$, showing that $u=0$ by the choice of $s$.

The first equation in \eqref{app_maxwell_equations} is satisfied by definition, and also the second equation is valid since by \eqref{app:sol_e_eq} and \eqref{app:e_claim} 
\begin{equation*}
*dH = \frac{1}{i\omega} \delta(\mu^{-1} dE) = -\frac{1}{i\omega} \eps^{-1} d\delta (\eps E) - i \omega \eps E = -i\omega \eps E.
\end{equation*}
The $1$-form $E$ is in $H^k_{\text{Div}}(M)$ and $tE = f$. The $1$-form $H$ is initially in $H^{k-1} \Omega^1(M)$ by definition. However, a similar argument which was used to establish \eqref{app_secondorder_e} shows that $h = \mu H$ satisfies the second order equation 
\begin{equation*}
\delta(\eps^{-1} d(\mu^{-1} h)) + \mu^{-1} d\delta h - \omega^2 h = 0.
\end{equation*}
Also, a computation in boundary normal coordinates gives the following boundary conditions for $h$:
\begin{align*}
t(*h) &= \frac{1}{i\omega} t(dE) = \frac{1}{i\omega} \text{Div}(f) \,dS, \\
t(\delta *h) &= t(*d(\mu H)) = t(*d\mu \wedge H) - i\omega \mu \eps f.
\end{align*}
Since $f \in T H^{k-1/2}_{\text{Div}}$ and $H \in H^{k-1}$, one can check by a computation in boundary normal coordinates that there exists $h^0 \in H^k \Omega^1(M)$ for which $\tilde{h} = h-h^0$ is in $H^{k-1}_A \Omega^1(M)$. Now $\tilde{h}$ satisfies the equation 
\begin{equation*}
\delta(\eps^{-1} d(\mu^{-1} \tilde{h})) + \mu^{-1} d\delta \tilde{h} - \omega^2 \tilde{h} = \tilde{F}
\end{equation*}
for some $\tilde{F} \in H^{k-2} \Omega^1(M)$. Elliptic regularity (as in Proposition \ref{app:wellposedness_secondorder_2}) implies that $\tilde{h} \in H^k \Omega^1(M)$ which is then true for $H$ too. We have $H \in H^k_{\text{Div}}(M)$ because 
\begin{equation*}
\text{Div}(tH) = -\langle \nu,*dH \rangle|_{\partial M} = i\omega\eps \langle \nu, E \rangle|_{\partial M} \in H^{k-1/2}(\partial M).
\end{equation*}
\end{proof}

\begin{remarks}
\begin{enumerate}
\item[1.] 
If $\eps, \mu \in C^2(M)$, then the conclusion of Theorem \ref{thm:app_wellposedness} is valid also for $k=1$. This follows from the above argument upon approximating $f \in T H^{1/2}_{\text{Div}}$ by smooth tangential fields.
\item[2.] 
Theorem \ref{thm:app_wellposedness} considers the case where $\eps$ and $\mu$ are independent of $\omega$. In applications the coefficients are often $\omega$-dependent, for instance in lossy materials one writes $\eps = \re(\eps) + i\sigma/\omega$ with $\re(\eps) > 0, \sigma \geq 0$, and $\mu > 0$ independent of $\omega$. In the last case, wellposedness in Euclidean domains was shown in \cite{SIC}. Of course, Theorems \ref{thm:main} and \ref{thm:main2} are valid whenever the admittance map is well defined (this is the content of assumption \eqref{paramcond3p}).
\end{enumerate}
\end{remarks}

\section{Unique continuation} \label{sec:ucp}

This section contains a unique continuation result for principally diagonal systems required in the final recovery of coefficients. The result is well known and follows from standard scalar Carleman estimates, but since we could not find a proper reference a proof is included here.

\begin{thm} \label{thm:app_uniquecontinuation}
Let $(M,g)$ be a compact connected Riemannian manifold with boundary, and let $\alpha_r, \beta_r$ be Lipschitz continuous functions in $M$ with positive real parts ($r=1,\ldots,N$). Consider the operators 
\begin{gather*}
P_r u = \frac{1}{\alpha_r} \delta(\beta_r du), \\
P = \diag(P_1, \ldots, P_N).
\end{gather*}
Let $\Gamma$ be an open subset of $\partial M$. If $\vec{u} \in H^2(M)^N$ satisfies 
\begin{gather*}
\abs{P \vec{u}(x)} \leq C (\abs{\vec{u}(x)} + \abs{\nabla \vec{u}(x)}) \ \ \text{for a.e.~$x \in M$}, \\
\vec{u}|_{\Gamma} = \partial_{\nu} \vec{u}|_{\Gamma} = 0,
\end{gather*}
then $\vec{u} \equiv 0$ in $M$.
\end{thm}

More generally, strong unique continuation holds in this setting. We will deduce Theorem \ref{thm:app_uniquecontinuation} from the next result which is stated in $\mR^n$.

\begin{thm} \label{thm:app:sucp}
Let $B$ be a ball in $\mR^n$ with center $x_0$, let $(g^{jk})_{j,k=1}^n$ be a Lipschitz continuous symmetric positive definite matrix in $B$, and let $\alpha_r, \beta_r$ be Lipschitz continuous functions in $B$ with positive real parts. Consider the operators 
\begin{gather*}
P_r u = \frac{1}{\alpha_r} \partial_{x_j} (\beta_r g^{jk} \partial_{x_k} u), \\
P = \diag(P_1, \ldots, P_N).
\end{gather*}
If $\vec{u} \in H^2(B)^N$ satisfies for all $K > 0$ 
\begin{gather*}
\abs{P \vec{u}(x)} \leq C (\abs{\vec{u}(x)} + \abs{\nabla \vec{u}(x)}) \ \ \text{for a.e.~$x \in B$}, \\
\lim_{r \to 0} r^{-K} \int_{B(x_0,r)} \abs{\vec{u}(x)}^2 \,dx = 0,
\end{gather*}
then $\vec{u} \equiv 0$ in $B$.
\end{thm}

\begin{proof}[Proof of Theorem \ref{thm:app_uniquecontinuation}]
If $\vec{u}$ is as in Theorem \ref{thm:app_uniquecontinuation}, we fix a point on $\Gamma$ and take $\tilde{M}$ to be a manifold obtained by enlarging $M$ slightly near this point. Extending $\vec{u}$ by zero to $\tilde{M}$ and extending $\alpha_r$ and $\beta_r$ as Lipschitz functions, we see that $\abs{P\vec{u}} \leq C(\abs{\vec{u}} + \abs{\nabla \vec{u}})$ a.e.~in $\tilde{M}$ and $\vec{u} = 0$ in some open subset. Working in local coordinates and using Theorem \ref{thm:app:sucp} with a connectedness argument proves the result.
\end{proof}

To prove Theorem \ref{thm:app:sucp}, note that we can assume $\alpha_r \equiv 1$, and by differentiation that $\beta_r \equiv 1$. Letting $Lu = \partial_{x_j}(g^{jk} \partial_{x_k} u)$, this implies that $P_1 = \ldots = P_N = L$. We may also assume that $x_0 = 0$ and $B = B(0,1)$, and that $\vec{u}$ is real valued. The result is a consequence of the following scalar Carleman estimate.

\begin{prop} \label{prop:app_ucp_ev}
Let $\lambda > 0$ be such that $\lambda^{-1} \abs{\xi}^2 \leq g^{jk} \xi_j \xi_k \leq \lambda \abs{\xi}^2$ in $B$ and $\sum_{j,k=1}^n \abs{g^{jk}(x)-g^{jk}(y)} \leq \lambda \abs{x-y}$ for $x, y \in B$. There exists $0 < \delta < 1$ and $M > 0$, only depending on $n$ and $\lambda$, and a function $w$ satisfying 
\begin{equation*}
\abs{x}/M \leq w(x) \leq M \abs{x} \quad \text{ in $B$},
\end{equation*}
such that for all $\alpha \geq M$ and all $u \in C^{\infty}_c(B(0,\delta) \smallsetminus \{0\})$ we have 
\begin{equation*}
\int_B (\alpha w^{1-2\alpha} \abs{\nabla u}^2 + \alpha^3 w^{-1-2\alpha} u^2) \,dx \leq M \int_B w^{2-2\alpha} (Lu)^2 \,dx.
\end{equation*}
\end{prop}

This estimate is proved in \cite[Theorem 2.1]{EV}, and is also contained in \cite{Hormander_ucp} with slightly different hypotheses. Theorem 2.1 in \cite{EV} is given in the parabolic setting, but the result above follows by applying the estimate in \cite{EV} to $v(x,t) = \theta(t) u(x)$ where $\theta$ is a cutoff function. One then notes that when the coefficients are independent of $t$, the $w(x,t)$ constructed in \cite{EV} depends only on $x$, so the inequality (2.1) in \cite{EV} yields the lemma by integrating in $t$ and absorbing the extra term on the right by making $\alpha$ even larger.

As an immediate corollary to Proposition \ref{prop:app_ucp_ev}, for $\vec{u} \in C^{\infty}_c(B(0,\delta) \smallsetminus \{0\})^N$ we have that 
\begin{equation*}
\int_B (\alpha w^{1-2\alpha} \abs{\nabla \vec{u}}^2 + \alpha^3 w^{-1-2\alpha} \abs{\vec{u}}^2) \,dx \leq M \int_B w^{2-2\alpha} \abs{P\vec{u}}^2 \,dx.
\end{equation*}
Once we have the last estimate, after the initial reductions, Theorem \ref{thm:app:sucp} follows immediately using the standard Carleman method.


\providecommand{\bysame}{\leavevmode\hbox to3em{\hrulefill}\thinspace}
\providecommand{\href}[2]{#2}

\end{document}